\documentclass[a4paper]{amsart}
\usepackage[english]{babel}
\usepackage{amssymb}
\usepackage{amsmath}
\usepackage{amsthm}
\usepackage{graphicx}
\usepackage{fancyhdr}
\usepackage{bbm}
\usepackage{pgfplots}

\newfont{\bssten}{cmssbx10}
\newfont{\bssnine}{cmssbx10 scaled 900}
\newfont{\bssdoz}{cmssbx10 scaled 1200}

\usepackage{latexsym,amssymb,amsthm,amsxtra}
\usepackage{amsmath}
\usepackage{stmaryrd}
\usepackage{mathrsfs}
\usepackage{tikz}
\usepackage{pgf}
\usetikzlibrary{topaths,calc}
{\bf}{\it}
{\bf}{\it}
\newtheorem{theorem}{Theorem}
\newtheorem{definition}{Definition}

\newtheorem{remark}{Remark}
\newtheorem{proposition}{Proposition}
\newtheorem{corollary}{Corollary}

\newtheorem{ex}{Example}

\def\/{\, | \,}

\def\ind{{\mathchoice {\rm 1\mskip-4mu l} {\rm 1\mskip-4mu l}
		{\rm 1\mskip-4.5mu l} {\rm 1\mskip-5mu l}}}
\def\N{{\mathbb N}}

\def\esp#1{{\mathbb E}\left[#1\right]}

\newcommand{\pae}[1]{\mbox{$\lfloor \kern-1pt #1 \kern-1pt \rfloor$}}
\newcommand{\paep}[1]{\mbox{$\lceil \kern-1pt #1 \kern-1pt \rceil$}}

\def\N{{\mathbb N}}
\def\R{{\mathbb R}}
\def\maV{{V}}

\def\maJ{{\mathcal J}}
\def\maH{{\mathcal H}}

\def\mbH{{\mathbb H}}

\def\mumin{\mu_{\mbox{\scriptsize{min}}}}
\def\mumax{\mu_{\mbox{\scriptsize{max}}}}
\def\mumini{\mu_{\mbox{\emph{\scriptsize{min}}}}}
\def\mumaxi{\mu_{\mbox{\emph{\scriptsize{max}}}}}

\newcommand\suite[1]{\left\{#1,\,n\in\N\right\}}

\newcommand\suiten[1]{\left\{#1,\,n\in\N\right\}}

\newcommand\gre{\textbf{e}}
\newcommand\grx{\textbf{x}}

\newcommand\maF{{\mathcal F}}
\newcommand\maB{{\mathcal B}}

\newcommand\maC{{\mathcal C}}
\newcommand\maE{{\mathcal E}}

\newcommand\maT{{\mathcal T}}

\usepackage[utf8]{inputenc}
\begin{document}
	
\title{A stochastic matching model on hypergraphs}
\author{Youssef Rahme and Pascal Moyal}

\begin{abstract}
Motivated by applications to a wide range of assemble-to-order systems, operations scheduling, healthcare systems and collaborative economy applications, 
we introduce a stochastic matching model on hypergraphs, extending the 
model in \cite{MaiMoy16} to the case of hypergraphical (rather than graphical) matching structures. 
We address a discrete-event system under a random input of single items, simply using the system as an interface to be matched by groups of two or more. 
We study the stability of this stochastic system, for various hypergraph geometries. 
\end{abstract}

\maketitle
\section{Introduction}
\label{sec:intro}
Matching models have recently received a growing interest in the literature on queueing models in which compatibilities between the requests need to be taken into account. 
This is a natural enrichment of service systems in which the requests must be matched, or put in relation, rather than being served. 
Among other fields of applications, this is a natural representation of peer-to-peer networks, interfaces of the collaborative economy (such as car and ride sharing, dating websites, 
and so on), assemble-to-order systems, job search applications and healthcare systems (blood banks and organ transplant networks). 
All these applications share the same common ground: elements/items/agents enter a system that is just an interface to put them in relation, and relations are possible only if the 
``properties" (whatever this means) of the elements make them compatible. 

A stochastic matching model can thus roughly be defined as follows: items enter a system at random times, and require 
to be matched by pairs. 
The possible pairs are given by a compatibility graph whose nodes represent the classes of items, 
and the classes of incoming items are randomly drawn from a prescribed 
distribution on the set of nodes. Unmatched items are queued, waiting for a future compatible item, and leave the system by pairs as soon as they are matched. 
If the items enter the system individually, as in \cite{MaiMoy16}, \cite{MoyPer17} and more recently in \cite{MBM17} and \cite{AKRW18}, we say that the system is a General stochastic Matching model 
(GM), following the terminology of \cite{MaiMoy16}. If the classes of items are partioned into two subsets (say the classes of ``customers'' and the classes of ``servers'') 
and enter the system pairwise, as in the seminal papers \cite{CKW09,AW11} (which viewed such systems as generalizations of skill-based customer/server queueing systems), and then \cite{BGM13}, \cite{ABMW17}, \cite{AKRW18} and \cite{MBM18}, we say that the system is a Bipartite stochastic Matching model (BM). 
Specific models for designated applications are studied: \cite{BDPS11} on kidney transplants, \cite{TW08} on housing allocations systems, 
\cite{BC15,BC17} on taxi hubs or \cite{OW19} on ride sharing models. In another line of research, such stochastic matching architectures are addressed from the point of view of stochastic optimization in \cite{BM14}, \cite{GW14} and \cite{NS16}, among others. 

In most of the above works (except in particular cases in \cite{GW14} and \cite{NS16}) the matching of items are exclusively {\em pairwise}: 
a job or a house with an applicant, a kidney with a patient, a cab with a customer, two users of a dating website, etc. 
However, several of the above applications should naturally incorporate the possibility of matching items 
by groups of more than two. Let us exemplify this on a concrete example: in organ transplants, (in)-compatibility between givers and receivers are given by a variety of factors, and 
mostly by blood types and immunological factors. 
In kidney exchange programs, items represent intra-incompatible couples $(A,B)$ 
(e.g. a patient $A$ waiting for a transplant and $B$ a parent of his/hers, incompatible with $A$ for a potential organ donation), 
entering a system to find another intra-incompatible couple $(A',B')$ that is compatible with it, in the sense that 
$A$ can receive an organ from $B'$ and $A'$ can receive from $B$. Then the ability of such a system to accommodate all requests and to maximize 
the number of successful transplants and avoid congestion, is translated into the positive recurrence of a stochastic process representing the stochastic system over time. 
Then if we view the items as the {\em couples}, and translate the ``cross-compatibility'' 
(i.e. $A$ can receive from $B'$ and $A'$ can receive from $B$) into the existence of an edge between node $(A,B)$ and node $(A',B')$, 
such a system is a typical application of the GM introduced in \cite{MaiMoy16}.

But let now consider the case where such exchanges $(A,B) \leftrightarrow (A',B')$ and $(A',B') \leftrightarrow (A'',B'')$ 
cannot be realized, but $A$ can receive from $B'$, $A'$ can receive from ${B''}$ and ${A''}$ can receive from $B$. 
Then it is natural to consider the possibility of executing the three transplants contemporarily, i.e. to match the 
triplet $(A,B)$, $(A',B')$, $(A'',B'')$ altogether. 
In several countries including the U.S., such ``exchanges'' by groups of 3 (or more) are allowed, which raises the issue of maximizing ``matchings'' that do not 
coincide with sets of edges, but of sets of subsets of nodes of cardinality $3$ or more. Hence the need to consider matching models on compatibility structures that are 
{\em hypergraphs} rather than graphs, i.e., a set of nodes $V$ equipped with a set of subsets of $V$ of cardinality 3 or more. 

Among other fields of applications, the same modeling is suitable to assemble-to-order systems, in which case components are produced by independent processes, and assembled by groups in a given order. All the same, in operations management, specific operations may be made available at given random times, to be coordinated later by 
groups of 2 or more. In all cases, the system controller confronts a random flux of arrivals of items (or operations), and needs to match (or combine/coordinate) them 
by sub-groups of 2 or more, hence following an hypergraphical structure. 

The main purpose of the present work is thus to introduce a stochastic matching model, in the sense defined above, on a hypergraphical compatibility structure. 
This model is formally defined as follows: items enter the system by single arrivals, and get matched by 
groups of 2 or more, following compatibilities that are represented by a given hypergraph. 
A matching policy determines the matchings to be executed in the case of a multiple choice, and the unmatched items are stored in a buffer, waiting for a future match. 

As for any dynamical random system, a first natural question to address is that of stochastic stability, i.e., the existence of a steady state: this step is necessary for investigating or 
comparing systems in the long run, in a stationary regime. In this paper, we formally define the stability region of the system as the region of measures on the set of nodes, rendering 
the natural Markov chain of the system positive recurrent, for a given compatibility hypergraph and a given matching policy. This paper is thus devoted to assessing the stability region 
of stochastic matching models on hypergraphs. In a nutshell, we show that such systems are not easily stabilizable, by exhibiting wide classes of models having an empty stability region. We then provide, or give bounds for, the stability region of particular systems. As will be made precise below, a crucial step is to investigate the very geometry of the hypergraphs under consideration. 

This paper is organized as follows: we start by some preliminary in Section \ref{sec:prelim}, and in particular by introducing the main definitions and properties of hypergraphs. 
In Section \ref{sec:model} we formally introduce the present model. In Section \ref{sec:Ncond} we provide necessary conditions of stability for the present class of systems: 
as will be developed therein, and unlike the particular case of the GM on graphs (see \cite{MaiMoy16}), for which a natural necessary condition could be obtained, we introduce 
various necessary conditions that depend on distinct geometrical properties of the considered hypergraphs. We then deduce from this, classes of hypergraphs for which the corresponding matching model cannot be stable, see Section \ref{sec:instable}. Finally, in Section \ref{sec:stable} we provide the precise stability region in the particular 
case where the compatibility hypergraph is $r$-uniform and complete, and then complete up to a partition of its hyperedges (see the precise definitions of these objects below). 
We conclude this work in Section \ref{sec:conclu}. 

\section{Preliminary}
\label{sec:prelim}

\subsection{General notation}
\label{subsec:notation}
Let $\R$, $\R^+$, $\N$ and $\N^+$ denote respectively the sets of real numbers, of non-negative real numbers, natural integers and positive integers, respectively. 
For $a$ and $b$ in $\N$, denote by $\llbracket a,b \rrbracket$ the integer interval $[a,b]\cap \N$.  We let $a\wedge b$ and $a\vee b$ denote respectively the minimum and the maximum of two numbers $a,b\in\R$.

Given a finite set $B$, we denote by $\mathscr M(B)$ the set of probability measures on $B$ having $B$ as exact support. 

Let $q\in \N^+$. For any $i\in \llbracket 1,q \rrbracket$, let $\gre_i$ denote the vector of $\mathbb{N}^{q}$ of components $(\gre_i)_j=\delta_{ij},\;j\in \llbracket 1,q \rrbracket$.  
The null vector of $\mathbb{N}^{q}$ is denoted by $\mathbf 0$. The norm of any vector $u \in \N^q$ is denoted by $\parallel u \parallel =\sum\limits_{i=1}^{q} u_i$. 

Let $A$ be a finite set. The cardinality of $A$ is denoted by $|A|$. We let $A^*$ denote the free monoid associated to $A$, i.e. the set of finite words over the alphabet $A$. 
The length of a word $w\in A^*$ is denoted by $|A|$. We write any word $w\in A^*$ as $w=w(1)w(2)...w(|w|)$. 
We denote for any $a\in A$, by $|w|_a$ the number of occurrences of letter $a$ in the word $w$. 
Having set an ordering on $A$, and denoting by $1,2,...,|A|$ the elements of $A$ in increasing order, the commutative image of a word $w\in A$ is the $\N^{|A|}$-valued vector $[w]$ defined by 
$[w]=\left(|w|_1,...,|w|_{|A|}\right)$ i.e. the vector whose $i$-th coordinate is the number of occurrences of letter $i$ in the word $w$. 
Finally, for a word $w\in A^*$ and an ordered lists of letters $(a,b,c,...)$ appearing in that order in $w$, we denote by $w\setminus_{(a,b,c,...)}$, the word of $A^*$ obtained by just deleting the letters $a,b,c...$ in $w$. 


\subsection{Hypergraphs}
\label{subsec:prelimhypergraphs}
For easy reference, let us first introduce the basics of Hypergraph theory that will be used in this paper. 
A thorough presentation of the topic can be found e.g. in \cite{Ber89}. 

\begin{definition}\label{Def:Hypergraph}
\rm
An hypergraph $\mathbb H$ is defined as a couple ($\maV$,$\maH$), where{:}
\begin{itemize}
\item The finite set $\maV$ is the set of nodes of $\mathbb H$. We let $q(\mbH)$ be the cardinality of $\maV$, and say that the hypergraph is of order $q(\mbH)$. 
\item A finite set $\maH:=\left\{H_1,...,H_{m(\mbH)}\right\}$ of subsets of $\maV$ such that $\bigcup_{i=1}^{m(\mbH)}H_i=\maV$, 
whose elements are called hyperedges of $\mbH$. 
\end{itemize}
We then say that the hypergraph is simple (or a {Sperner family}) if $H_i\subset H_j$ implies $i=j$ for all $i,j \in \llbracket 1,m(\mbH) \rrbracket$, i.e.,  
no hyperedge is included in another one. Whenever no ambiguity is possible, we often write $q:=q(\mbH)$, $m:=m(\mbH)$. 
A sub-hypergraph of $\mbH$ is an hypergraph $\mbH'=(V,\maH')$ such that $\maH' \subset \maH$. 
\end{definition}

\begin{definition}\label{Def:rankdegree}
\rm
Let $\mathbb H=(\maV,\maH)$ be an hypergraph. 
The {rank} of $\mathbb H$ is the largest size of an hyperedge, i.e. the integer $r(\mathbb H)=\max_{j\in\llbracket 1,m(\mbH)\rrbracket}|H_j|$; 
the {anti-rank} of $\mathbb H$ is defined as $a(\mathbb H)=\min_{j\in\llbracket 1,m(\mbH)\rrbracket}|H_j|$, i.e. the smallest size of an hyperedge. 
If there exists a constant $r$ such that $r(\mathbb H)= a(\mathbb H)=r$, then $\mathbb H$ is said $r$-{uniform}. 
The degree of a node $i\in \maV$ is the number of hyperedges $i$ belongs to, i.e. $d(i)=\sum_{\ell=1}^{m(\mbH)}\ind_{H_{\ell}}(i)$. If there exists a constant $d$ such that $d(i)=d$ for any $i$, then $\mathbb H$ is said $d$-{regular}.
\end{definition}

\begin{definition}\label{Def:GraphRepresentatif}
\rm 
The {representative graph} of an hypergraph $\mathbb H=(\maV,\maH)$ is the graph $L(\mathbb H)=(\maH,\maE)$ whose nodes are the elements of $\maH$, 
and such that $(H_i,H_j) \in \maE$ (i.e. $H_i$ and $H_j$ share an edge in the graph) if and only if $H_i\cap H_j\neq\emptyset$. 
The hypergraph $\mathbb H$ is said {connected} if $L(\mathbb H)$ is connected. 
\end{definition}

As is easily seen, any $2$-uniform hypergraph is a graph, whose edges are the elements of $\maH$, and any simple and connected hypergraph contains no isolated node, i.e. has anti-rank at least 2. 


\begin{definition}
\label{def:transverse}
\rm
A set $T\subset \maV$ is a {transversal} of
$\mathbb{H}$ if it meets all its {hyper}edges, that is, $T\cap H \neq\emptyset, \mbox{ for any }H \in \maH.$ 
The set of transversals of $\mbH$ is denoted by $\maT(\mbH)$. 
A transversal $T$ is said minimal if it is of minimal cardinality among all transversals of $\mbH$. 
The transversal number of the hypergraph $\mbH$ is the cardinality of its minimal transversals. It is denoted $\tau(\mbH)$. 
\end{definition}



For any set $A \subset \maV$, we denote 
\begin{equation}
\label{eq:defHA}
\maH(A) = \left\{H \in \maH\;:\; H \cap A\ne \emptyset\right\},
\end{equation}
i.e. the set of hyperedges that intersect with $A$. With some abuse, for any node $i\in\maV$, we write $\maH(i):=\maH(\{i\})$. 

Throughout this paper, all considered hypergraphs are simple and connected.

\section{The model}
\label{sec:model}

All the random variables (r.v.'s, for short) hereafter 
are defined on a common probability space $(\Omega,\maF,\mathbb P)$. 

\subsection{Stochastic matching model on a hypergraph}
\label{subsec:model}
A (discrete-time, hypergraphical) stochastic matching model is specified by a triple $(\mathbb H,\Phi,\mu)$, such that:
\begin{itemize}
\item $\mathbb H=(\maV,\maH)$ is a simple and connected hypergraph, termed {\em matching hypergraph} of the model,
\item $\Phi$ is a matching policy, precisely defined in section \ref{subsec:Policy} below, 
\item $\mu$ is an element of $\mathscr M(\maV)$.
\end{itemize}

\noindent The matching model $(\mathbb H,\Phi,\mu)$ is then defined as follows. At each time point $n\in\N$, 
\begin{enumerate}
\item An item enters the system. Its class $V_n$ is drawn from the measure $\mu$ on $\maV$, independently of everything else. 
(Thus the sequence of classes of incoming items $\suite{V_n}$ is i.i.d of common distribution $\mu$.)
\item The incoming item then faces the following alternative: 
        \begin{itemize}
\item[(i)] If there exists in the buffer, at least one set of items whose respective set of classes forms, together with $V_n$, an hyperedge of $\maH$, then it is the role of the matching policy $\Phi$ to select one of these sets of classes, say 
         $\{i_1,...,i_m\}$. Then the $m+1$ items of respective classes $i_1,...,i_m,V_n$ are matched 
         together and leave the system right away. Denoting  $H_j:=\{i_1,...,i_m,V_n\} \in \maH$, we then say that $V_n$ {\em completes} a {\em matching} of type $H_j$ at time $n$, and we denote $H(n)=H_j$, the matching performed at $n$. 
\item[(ii)] {E}lse, the item is stored in the buffer of the system, waiting for a future match, and we write $H(n)=\emptyset$. 
\end{itemize}
\end{enumerate}

\begin{ex}
\rm 
Consider the matching hypergraph $\mathbb{H}=(\maV,\maH)$ with $\maV=\{1,2,3,4\}$ and $\maH=\left\lbrace\{1,2,3\},\{1,2,4\},\{1,3,4\},\{2,3,4\}\right\rbrace$, 
see Figure \ref{fig:complete}. The dynamic matchings of the realization $\suite{V_n(\omega)}=2,3,4,1,1,2,3,3,4,2,2,....$ is represented in Figure \ref{fig:Dyn}. 

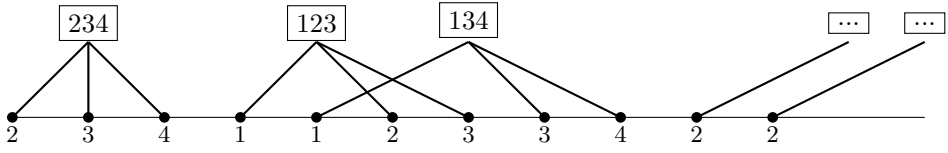
\begin{figure}[htp]
	\begin{center}
		\begin{tikzpicture}
		\draw[-] (1,2) -- (13,2);
		\fill (1,2) circle (2pt) node[below] {\small{2}};
		\node[draw] at (2,3.25) {234};
		\draw[-, thick] (1,2) -- (2,3);
		\draw[-, thick] (2,2) -- (2,3);
		\draw[-, thick] (3,2) -- (2,3);
		\fill (2,2) circle (2pt) node[below] {\small{3}};
		\fill (3,2) circle (2pt) node[below] {\small{4}};
		\fill (4,2) circle (2pt) node[below] {\small{1}};
		\node[draw] at (5,3.25) {123};
		\draw[-, thick] (4,2) -- (5,3);
		\draw[-, thick] (6,2) -- (5,3);
		\draw[-, thick] (7,2) -- (5,3);
		\fill (5,2) circle (2pt) node[below] {\small{1}};
		\fill (6,2) circle (2pt) node[below] {\small{2}};
		\fill (7,2) circle (2pt) node[below] {\small{3}};
		\fill (8,2) circle (2pt) node[below] {\small{3}};
		\fill (9,2) circle (2pt) node[below] {\small{4}};
		\fill (10,2) circle (2pt) node[below] {\small{2}};
		\node[draw] at (7,3.29) {134};
		\draw[-, thick] (5,2) -- (7,3);
		\draw[-, thick] (8,2) -- (7,3);
		\draw[-, thick] (9,2) -- (7,3);
		\fill (11,2) circle (2pt) node[below] {\small{2}};
		\node[draw] at (12,3.25) {...};
		\node[draw] at (13,3.25) {...};
		\draw[-, thick] (10,2) -- (12,3);
		\draw[-, thick] (11,2) -- (13,3);
		\end{tikzpicture}
		\caption[smallcaption]{\label{fig:Dyn} The matching model in action, on the matching hypergraph of Figure \ref{fig:complete}.}
	\end{center}
\end{figure}


\end{ex}

\subsection{System dynamics}
\label{subsec:dyn}
Fix a hypergraphical matching model on a hypergraph $\mbH$ of order $q:=q(\mbH)$. 
Define for all $n\in\N$ the $\N^{q}$-valued r.v. $X_n=\left(X_n(1),...,X_n\left(q\right)\right),$
where for all $i \in \maV$, $X_n(i)$ is the number of items of class $i$ in the buffer at time $n$ (taking into account the arrival occurring at time $n$).  
The vector $X_n$ is then called {\em class-content} of the system at time $n$. 
Define for any subset $B$ of $\maV$, $X_n(B)$ to be the class-content of elements of $B$: 
\[X_n(B)=\sum_{i\in B}X_n(i),\]
in a way that the total number of items in the buffer at time $n$ is given by $\parallel X_n \parallel$. 
The {\em buffer-content} of the system at time $n$ is the word of $V^*$ whose letters are the classes of the items in the buffer, in increasing order of their arrivals. Namely, 
\[W_n = W_n(1) W_n(2) ..... W_n(|W_n|),\]
where for any $\ell$, $W_n(\ell)$ is the class of the $\ell$-th oldest item in line. Notice that $X_n$ is nothing but the commutative image of $W_n$, i.e. 
$X_n=[W_n],\,n\in \N.$ 
 
To simply describe the dynamics of the processes $\suite{X_n}$ and $\suite{W_n}$, for any $u\in \N^q$ and $H \in \maH$ we define the following elements of $\{0,1\}^q$: 
$p(u)$ is the vector of coordinates $p(u)_i=\ind_{\{u(i)>0\}}$, $i\in \llbracket 1,q \rrbracket$, $\gamma(H)$ is the {trace} of $H$, i.e. the vector of $\{0,1\}^q$ defined by 
$\gamma(H)_i = \ind_{H}(i)$ for all $i \in \llbracket 1,q \rrbracket$ and 
\[\Gamma(u) =  \left\{H \in \maH:\, p\left(u\right)=\gamma(H)\right\}.\]


\subsection{Matching policies}
\label{subsec:Policy}
Formally, an admissible matching policy is a rule of choice of the item(s) matched with the incoming item at time $n+1$ in case of a multiple choice, that can be made solely on the basis 
of the knowledge of the buffer-content $W_n$ at $n$, for any $n\in\N$. Notice that at all $n$, 
$\Gamma\left(X_n+e_{V_{n+1}}\right)=\Gamma\left([W_n]+e_{V_{n+1}}\right)$ represents the (possibly empty) set of all hyperedges in $\maH$ that can be completed by the arrival of $V_{n+1}$ in 
a system having buffer-content $W_n$ at time $n$. 

\subsubsection{Matching policies that depend on the arrival times}

\subsubsection*{First Come, First Matched}
In First Come, First Matched ({\sc fcfm}), the chosen match of the incoming item $V_{n+1}$ at time $n+1$, is the hyperedge containing the oldest item in line among 
all hyperedges that can be completed by $V_{n+1}$.  
Specifically: 
\begin{equation*}
W_{n+1}=\left\lbrace \begin{array}{ll}
W_nV_{n+1}& \qquad \textrm{if}\;\Gamma\left([W_n]+e_{V_{n+1}}\right)=\emptyset;\\
W_nV_{n+1}\setminus_{(i,j,...,k)}& \qquad \mbox{else, for }H(n)=\{i,j,...,k\}\in \Gamma\left([W_n]+e_{V_{n+1}}\right),
\end{array}
\right.
\end{equation*}
where, in the case where $\left |\Gamma\left(X_n+e_{V_{n+1}}\right) \right| \ge 2$, i.e. there are more than one possible matchings containing $V_{n+1}$ at $n+1$, 
$H(n)=\{i,j,...,k\}$ is the hyperedge whose first element $i$ appearing in $W_n$ appears first among all elements of $\Gamma\left([W_n]+e_{V_{n+1}}\right).$ 

\subsubsection*{Last Come, First Matched}
Likewise, in Last Come, First Matched ({\sc lcfm}) the newly arrived item at $n+1$ is matched to form the hyperedge containing the youngest possible element, i.e. 
\begin{equation*}
W_{n+1}=\left\lbrace \begin{array}{ll}
W_nV_{n+1}& \qquad \textrm{if}\;\Gamma\left([W_n]+e_{V_{n+1}}\right)=\emptyset;\\
W_nV_{n+1}\setminus_{(i,j,...,k)}& \qquad \mbox{else, for }H(n)=\{i,j,...,k\} \in \Gamma\left([W_n]+e_{V_{n+1}}\right),
\end{array}
\right.
\end{equation*}
where $H(n)=\{i,j,...,k\}$ is the hyperedge whose last element $k$ appearing in $W_n$ appears last among all elements of $\Gamma\left([W_n]+e_{V_{n+1}}\right).$ 

\subsubsection{Matching policies that depend on the class-content} 

A wide class of natural matching policies can be implemented given the sole knowledge of the class-content upon arrival times. In such cases we have for all $n$, 
\begin{equation}
\label{buffer:commutative}
X_{n+1}=\left\lbrace \begin{array}{ll}
X_n+e_{V_{n+1}} & \qquad \textrm{if}\;\Gamma\left(X_n+e_{V_{n+1}}\right)=\emptyset;\\
X_n+e_{V_{n+1}} - \gamma(H(n))& \qquad \mbox{else, for some }H(n)\in \Gamma\left(X_n+e_{V_{n+1}}\right),
\end{array}
\right.
\end{equation}
where the choice of the hyperedge $H(n)$ depends on the matching policy. 
Several examples are provided below,

\subsubsection*{Match the Longest.} 
The matching policy $\Phi$ is {\em Match the longest} (denoted $\textsc{ml})$ if for all $n$, the match realized is that 
of the hyperedge having the most elements in storage at $n$. In other words the chosen hyperedge $H(n)$ 
in the second case of (\ref{buffer:commutative}) satisfies 
\[H(n) = \mbox{argmax\,}\left\{X_n(H)\,:\,H \in \Gamma\left(X_n+e_{V_{n+1}}\right)\right\},\]
ties being broken uniformly at random (and independently of everything else) among hyperedges. 

\subsubsection*{Match the Shortest.} 
Analogously, {\em Match the shortest} (denoted $\textsc{ms})$ corresponds to the choice 
\[H(n) = \mbox{argmin\,}\left\{X_n(H)\,:\,H \in \Gamma\left(X_n+e_{V_{n+1}}\right)\right\},\]
ties being broken uniformly at random, as above.  

\subsubsection*{Fixed priority.} 
In the context of fixed priorities, each vertex $i \in \maV$ is assigned a full ordering of the hyperedges and choses to be matched with the first matchable 
hyperedge following this order. Formally, to each node $i$ is associated a permutation $\sigma_i$ of the index set $\llbracket 1,d(i) \rrbracket$, and if we denote $\maH(i)=\left\{H_{i_1},H_{i_2},...,H_{i_{d(i)}}\right\}$, 
then at any time $n$, 
\begin{equation}
\label{eq:defpriority}
H(n) = H_{i_{\sigma_i(j)}},\mbox{ where }j=\mbox{min\,}\left\{k\in \llbracket 1,d(i) \rrbracket\,:\,X_n\left(H_{i_{\sigma(k)}}\right)>0\right\}.
\end{equation}

\subsubsection*{Random.} 
For this matching policy, the priority order defined above is not fixed, and is drawn uniformly at random upon each arrival, i.e. for any $n$, 
$H(n)$ is defined as in (\ref{eq:defpriority}), for a permutation $\sigma_i(n)$ that is drawn, independently of everything else, uniformly at random 
among all permutations of $\llbracket 1,d(i) \rrbracket$. 

\medskip

It is easily seen that under any admissible policy the sequence $\suite{W_n}$ is a $V^*$-valued Markov chain with respect to the filtration generated by the sequence $\suite{V_n}$. 
Additionally, in the cases where $\Phi=$ {\sc ml}, {\sc ms}, a fixed priority or a random policy, the sequence $\suite{X_n}$ is a $\N^{|V|}$-valued Markov chain.

\subsection{Stability of the matching model}
\label{subsec:defstab}
We say that the matching model $(\mathbb H,\Phi,\mu)$ is stable if the Markov chain $\suite{W_n}$ (and thereby $\suite{X_n}$) is positive recurrent. 
For a given hypergraph $\mathbb H=(\maV,\maH)$ and a given matching policy $\Phi$, we define 
the {\em stability region} associated to $\mathbb H$ and $\Phi$ as the set of probability measures on $\maV$ rendering the model $(\mathbb H,\Phi,\mu)$ stable, i.e. 
\[\textsc{Stab}(\mathbb H,\Phi)=\left\{\mu \in\mathscr M(\maV):\;(\mathbb H,\Phi,\mu)\mbox{ is stable }\right\}.\]
We then say that an hypergraph $\mbH$ is stabilizable if $\textsc{Stab}(\mathbb H,\Phi)$ is non-empty for some matching policy $\Phi$. 
If not, $\mbH$ is said non-stabilizable.

\section{Necessary conditions of stability}
\label{sec:Ncond} 

Fix a matching model $(\mathbb H,\Phi,\mu)$ on an hypergraph $\mathbb H=(\maV,\maH)$. 
Denote for any $n$, $B \subset V$ and $\maB \subset \maH$, by $A_n(B)$ the number of arrivals of elements in $B$ and by $M_n(\maB)$ the number of matchings of hyperedges in $\maB$ realized up to $n$, i.e. 
\begin{align*}
A_n(B) &= \sum_{k=1}^n \ind_{\{V_k \in B\}};\\
M_n(\maB) &= \sum_{k=1}^n \ind_{\{H(k) \in \maB\}},
\end{align*}
and with some abuse, denote $A_n(i)=A_n(\{i\})$ and $M_n(H)=M_n(\{H\})$ for any $i\in V$ and $H \in \maH$. 
Observe that the following key relation holds for all $B \subset \maV$, 
\begin{equation}
X_n(B) = A_n(B) - \sum\limits_{H \in \maH} \left|H \cap B \right| M_n\left(H\right)\ge 0,\,\,n\in\N, \label{eq:base}
\end{equation}
since the number of items of classes in $B$ at any time $n$ is precisely the number of arrivals of such items up to time $n$, minus 
the number of these items that leave the system upon each matching of an hyperedge that intersects with $B$.

\subsection{General conditions}
\label{subsec:Ncond}
\noindent We start by introducing several `universal' stability conditions. 
Fix an hypergraph $\mathbb H=(\maV,\maH)$ throughout the section, and let us define the set 
\begin{equation*}
\maC_2(\mbH) = \left\{B\subset V\,:\,\max_{H\in\maH}|B\cap H| \ge 2\right\}, 
\end{equation*}
and the following sets of measures,
\begin{align*}
 \mathscr N^{\scriptsize{+}}_1(\mathbb H) &= \left\{\mu\in\mathscr M(\maV):\,\forall B \subset V,\,\mu(B)\;\le \sum_{H\in \maH}\left|H \cap B\right|\min_{j \in H}\mu(j)\right\};\\
\mathscr N^{\scriptsize{-}}_1(\mathbb H) &= \left\{\mu\in\mathscr M(\maV):\,\forall B \in \maC_2(\mbH),\,\mu(B)\;<\sum_{H\in \maH}\left|H \cap B\right|\min_{j \in H}\mu(j)\right\};\\
 \mathscr N^{\tiny{--}}_1(\mathbb H)&= \left\{\mu\in\mathscr M(\maV):\, \forall B \subset V,\,\mu(B)\;<\sum_{H\in \maH}\left|H \cap B\right|\min_{j \in H \cap \bar B}\mu(j)\right\}.
\end{align*}
We have the following result, 
\begin{proposition}
\label{prop:Ncond1}
For any connected hypergraph $\mathbb H$ and any admissible matching policy $\Phi$,
\begin{equation*}
\textsc{Stab}(\mbH,\Phi) \subset \mathscr N^{\scriptsize{+}}_1(\mathbb H)\cap \mathscr N^{\scriptsize{-}}_1(\mathbb H) \cap \mathscr N^{\tiny{--}}_1(\mathbb H).
\end{equation*}
\end{proposition}

\begin{proof}
Fix $\mbH=(V,\maH)$, and let us denote for any $\mu\in\mathscr M(V)$ and $B \subset V$, 
\[\ell_\mu(B) = \mbox{argmin} \left\{\mu(j)\,:\,j\in B\right\},\] 
where $\ell_\mu(H)$ is chosen arbitrarily whenever the above is not unique.  
Fix an admissible policy $\Phi$. 
We first prove the inclusion of $\textsc{Stab}(\mbH,\Phi)$ in $\mathscr N^{\scriptsize{+}}_1(\mathbb H)$. 
For this, suppose that $\mu \in \mathscr M(V)$ is such that there exists $B\subset V$ such that 
\begin{equation}
\label{eq:contrNcond1}
\mu(B)\;> \sum\limits_{H\in\maH}\left|H \cap B\right|\mu\left(\ell_\mu(H)\right).
\end{equation}

Then, for any $H \in\maH$ and any $n\in\N$ we have that 
$M_n(H) \le \min_{j \in H} A_n(j) \le A_n\left(\ell_\mu(H)\right)$ and thus,  
from the left equality in (\ref{eq:base}) that 
\begin{equation}
\label{eq:compareXY}
{X_n(B) \over n} \ge {A_n(B) \over n} - \sum\limits_{H\in\maH}\left|H \cap B\right| {A_n\left(\ell_\mu(H)\right) \over n}.
\end{equation}
Applying the SLLN to the right-hand side of (\ref{eq:compareXY}) implies that 
\[\limsup_n {X_n(B) \over n}\ge \mu(B) - \sum\limits_{H\in\maH}\left|H \cap B\right| \mu\left(\ell_\mu(H)\right) >0,\] 
implying that $X_n(B)$ goes a.s. to infinity and thereby (as $X_n=[W_n]$ for all $n$), the transience of $\suite{W_n}$. 

\medskip

Regarding the inclusion in $\mathscr N_1^{\scriptsize{-}}(\mathbb H)$, suppose now that $\mu$ is such that there exists $B\in \maC_2(\mbH)$ such that there is an equality in (\ref{eq:contrNcond1}). 
Then the Markov chain $\suite{Y_n}$ defined as
\begin{equation*}
Y_n =\left(A_n(B)\;-\;\sum\limits_{H\in\maH}\left|H \cap B\right|A_n\left(\ell_\mu(H)\right)\right),\quad n\in\N,
\end{equation*}
is a random walk with drift 0, that is different from the identically null process, since it makes down jumps of at least 1 
at any arrival time of an element of any $H$ such that $|B\cap H| \ge 2$. 
Hence $\suite{Y_n}$ is null recurrent. Would the chain $\suite{W_n}$ be positive recurrent, the sequence 
$\suite{X_n}$ would visit the state 
$\mathbf 0$ infinitely often, with inter-passage time at $\mathbf 0$ of finite expectation. 
Thus from (\ref{eq:compareXY}), the sequence $\suite{Y_n}$ would be positive recurrent, an absurdity. 
Thus the stability region is included in 
$\mathscr N_1^{\scriptsize{-}}(\mathbb H)$. 

\medskip

Last, to show the inclusion in $\mathscr N_1^{\tiny{--}}(\mathbb H)$, we let $\mu\in\mathscr M(V)$ be such that for some $B\subset V$, 
\begin{equation}
\label{eq:contrNcond2}
\mu(B)\;\ge \sum\limits_{H\in\maH}\left|H \cap B\right|\min_{j \in H\cap \bar B} \mu(j). 
\end{equation}
If the inequality above is strong, then (\ref{eq:contrNcond1}) holds, implying the transience of $\suite{W_n}$. 
If now the equality in (\ref{eq:contrNcond2}) is weak, then 
the Markov chain $\suite{\bar Y_n}$ defined by 
\begin{equation*}
\bar Y_n =\left(A_n(B)\;-\;\sum\limits_{H\in\maH}\left|H \cap B\right|A_n\left(\ell_\mu(H\cap \bar B)\right)\right),\quad n\in\N,
\end{equation*}
is a non zero random walk with null drift, and we conclude as above. 
\end{proof}




Let us now define the following set of measures, 
\begin{equation*}
 \mathscr N_2(\mathbb H)= \left\{\mu\in\mathscr M(\maV):\,\quad \forall T \in \maT(\mathbb H)\;,\,\mu(T)\;>{1\over r(\mbH)}\right\}.
 \end{equation*}
We also have that
\begin{proposition}
\label{prop:Ncond2}
For any connected hypergraph $\mathbb H$ and any admissible matching policy $\Phi$,
\begin{equation*}
\textsc{Stab}(\mbH,\Phi) \subset \mathscr N_2(\mathbb H).
\end{equation*}
\end{proposition}
\begin{proof}
Suppose that there exists a transversal $T\in\maT(\mbH)$ such that $\mu(T)\;\le{1\over r(\mbH)}$. 
It is then easily seen that $M_n(\maH)\leq A_n(T)$  for all $n$. Thus, for all $n$ we have that 
		\begin{equation*}
		{X_n(V)\over n}\ge {1\over n}\left(A_n(V)-r(\mbH)M_n(\maH)\right)\ge {1\over n}\left(A_n(V)-r(\mbH)A_n(T)\right).
		\end{equation*}
		Taking $n$ to infinity in the above yields 
		\begin{equation*}
		\limsup_n {X_n(V)\over n} \ge 1 - r(\mbH)\mu(T),
		\end{equation*}
and we conclude as in the previous proof. 
	\end{proof}
	
\begin{remark}
\label{rem:unifsubN2}
\rm
As an immediate consequence of Proposition \ref{prop:Ncond2}, if $\mbH=(V,\maH)$ is of {order} $q$, 
and such that $\tau(\mbH) \le {q \over r(\mbH)}$, 
then $\textsc{Stab}(\mbH,\Phi)$ does not contain the uniform measure $\mu_{\textsc{u}}=(1/q,...,1/q)$ on $V$, in other words the model $(\mbH,\Phi,\mu_{\textsc{u}})$ is instable 
for any $\Phi$. Indeed, for any minimal transversal $T$ of $\mbH$ we have that 
\[\mu_{\textsc{u}}(T) = {\tau(\mbH) \over q} \le {1\over r(\mbH)}.\]
\end{remark}

\medskip

We now introduce two necessary conditions of stability based on the anti-rank of the considered hypergraph. 
We first introduce the following sets of measures, 
\begin{align}
\label{eq:Ncond3+}
 \mathscr N^{\scriptsize{+}}_3(\mathbb H) &= \left\{\mu \in \mathscr M(\maV):\;\forall i \in \maV,\,\mu(i) \le {1 \over a(\mbH)}\right\};\\
\label{eq:Ncond3-}
 \mathscr N^{-}_3(\mathbb H)&= \left\{\mu \in \mathscr M(\maV):\;\forall i \in \maV,\,\mu(i) < {1 \over a(\mbH)}\right\}. 
\end{align}
We have the following, 
\begin{proposition}
\label{prop:Ncond3}
For any connected hypergraph $\mathbb H=(\maV,\maH)$ and any admissible policy $\Phi$, 
\begin{equation}
\label{eq:Ncondantirank}
\textsc{Stab}(\mathbb H,\Phi) \subset \mathscr N^{\scriptsize{+}}_3(\mathbb H). 
\end{equation}
If the hypergraph $\mathbb H=(\maV,\maH)$ is $r$-{uniform} (i.e. $a(\mbH)=r(\mbH)=r$) we have that 
\begin{equation}
\label{eq:Ncondkunifcomplet}
\textsc{Stab}(\mathbb H,\Phi) \subset \mathscr N^{-}_3(\mathbb H).
\end{equation}
in other words the model $(\mbH,\Phi,\mu)$ cannot be stable unless $\mu(i)<1/r$ for any $i\in V$. 
\end{proposition}

\begin{proof}
To prove the first statement, we argue again by contradiction. Suppose that 
$\mu(i_0) > {1 \over a(\mbH)}$ for some node $i_0$. 
As the function 
\[
\begin{cases}
\R_+ &\longrightarrow \R_+\\
x &\longmapsto {r(\mbH)-a(\mbH)+x \over xa(\mbH)}
\end{cases}\]
strictly decreases to ${1\over a}$, there exists $x_0 > 0$ such that 
\begin{equation}
\label{eq:i0}
\mu(i_0) >  {r(\mbH)-a(\mbH)+x_0 \over x_0a(\mbH)}.
\end{equation}
Then, applying the right inequality in (\ref{eq:base}) to $B \equiv \maV \setminus \{i_0\}$, we readily obtain that a.s. for all $n$, 
\begin{multline}
{r(\mbH)+x_0\over a(\mbH)}A_n\left(\maV\setminus\{i_0\}\right)\\
\begin{aligned}
&\ge {r(\mbH)+x_0\over a(\mbH)}\left(\sum\limits_{H \in \maH(i_0)} \left|H -1\right| M_n\left(H\right)+\sum\limits_{H \in \overline{\maH(i_0)}} \left|H \right| M_n\left(H\right)\right)\\
&\ge \left(r(\mbH)+x_0-{r(\mbH)+x_0\over a(\mbH)}\right)M_n\left(\maH(i_0)\right) + (r(\mbH)+x_0)M_n\left(\overline{\maH(i_0)}\right). \label{eq:conditionM}
\end{aligned}
\end{multline}
Likewise, applying the left equality of (\ref{eq:base}) to $\{i_0\}$ and then $\maV\setminus \{i_0\}$ also yields to 
\begin{multline*}
X_n\left(\maV\setminus \{i_0\}\right)+\left(x_0 +1 - {r(\mbH)+x_0 \over a(\mbH)}\right)X_{n}(i_0)\\
\shoveleft{
= A_n\left(\maV\setminus \{i_0\}\right)-\sum\limits_{H \in \maH(i_0)} \left| H -1\right| M_n\left(H\right)-\sum\limits_{H \in \overline{\maH(i_0)}} \left| H\right| M_n\left(H\right)}\\
\shoveright{+\left(x_0 +1- {r(\mbH)+x_0 \over a(\mbH)}\right)\left(A_n(i_0)-M_n\left(\maH(i_0)\right)\right)}\\
\shoveleft{> A_n\left(\maV\setminus \{i_0\}\right)+\left(x_0 +1- {r(\mbH)+x_0 \over a(\mbH)}\right)A_n(i_0)}\\
-\left(r(\mbH)+ x_0 - {r(\mbH)+x_0 \over a(\mbH)}\right)M_n(\maH(i_0)) - (r(\mbH)+x_0)M_n\left(\overline{\maH(i_0)}\right).
\end{multline*}
Combining this with (\ref{eq:conditionM}), implies that a.s. for all $n$, 
\begin{equation*}
X_n\left(\maV\right)+\left(x_0 - {r(\mbH)+x_0 \over a(\mbH)}\right)X_{n}(i_0) 
>  \left(1- {r(\mbH)+x_0 \over a(\mbH)}\right)A_n\left(\maV\right)+x_0A_n(i_0).
\end{equation*}
Therefore we have that 
\begin{equation}
\label{eq:boundN3}
\limsup_n {1\over n}\left(X_n\left(\maV\right)+\left(x_0 - {r(\mbH)+x_0 \over a(\mbH)}\right)X_{n}(i_0)\right) \ge 
 1- {r(\mbH)+x_0 \over a(\mbH)}+x_0\mu(i_0),
 \end{equation}
hence the chain $(W_n)$ is transient since the right-hand side of the above is positive from (\ref{eq:i0}).  

\medskip
It remains to check that in the case where the hypergraph is $r$-uniform, the model cannot be stable whenever 
$\mu(i_0) \ge {1\over a(\mbH)}={1\over r}$ 
for some $i_0\in V$. For this, notice that, as $r(\mbH)=a(\mbH)=r$ a weak inequality holds true in (\ref{eq:i0}) for any $x_0>0$. 
Then, it readily follows from (\ref{eq:boundN3}) that for any $x_0$, 
$$\limsup_n {1\over n}\left(X_n\left(\maV\right)+\left(x_0 - {r(\mbH)+x_0 \over a(\mbH)}\right)X_{n}(i_0)\right) \ge 0,$$
and we conclude, as in the proof of Proposition \ref{prop:Ncond1}, that the chain $(W_n)$ is at best null recurrent. 
\end{proof}


\section{{Non-stabilizable hypergraphs}}
\label{sec:instable} 
Having Propositions \ref{prop:Ncond1}, \ref{prop:Ncond2} and \ref{prop:Ncond3} in hand, one can identify classes of hypergraphs 
$\mbH$ such that $(\mbH,\Phi,\mu)$ has an empty stability region for any admissible $\Phi$. 

We start with the following elementary observation, 
\begin{proposition}
	\label{prop:isolated}
	If an hyperedge of $\mbH=(V,\maH)$ contains two isolated nodes, i.e. there exist $H\in \maH$ and $i,j \in H$ such that $d(i)=d(j)=1$, then 
	the model cannot be stable, i.e. $\textsc{Stab} (\mbH,\Phi)=\emptyset$ for any admissible $\Phi$ . 
\end{proposition}

\begin{proof}
	Let $\mu \in  \mathscr N^{\tiny{--}}_1(\mbH)$. Then, considering successively the sets $\{i\}$ and $\{j\}$, 
	as $j \in H\cap \bar{\{i\}}$ and $i \in H \cap \bar{\{j\}}$ we obtain that $\mu(i) < \mu(j)\mbox{ and }\mu(i) > \mu(j),$ an absurdity. 
\end{proof}

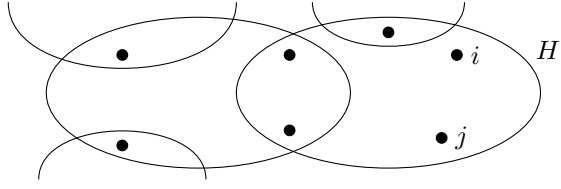
\begin{figure}[htp]
	\def\firstellip{(1, 2) ellipse [x radius=2cm, y radius=1cm, rotate=180]}
	\def\secondellip{(3.5, 2cm) ellipse [x radius=2cm, y radius=1cm, rotate=180]} 
	\def\thirdellip{(-1.5, 2) ellipse [x radius=2cm, y radius=1cm, rotate=-180]}
	\begin{tikzpicture}
	
	
	\filldraw (-2.5,1.3) circle (2pt) node [right] {};
	\draw[-] (-3.6,0.85) to[bend left=90] (-1.4,0.85);
	\filldraw (-2.5,2.5) circle (2pt) node [right] {};
	\draw[-] (-4,3.2) to[bend right=90] (-1,3.2);
	\filldraw (-0.3,2.5) circle (2pt) node [right] {};
	\filldraw (-0.3,1.5) circle (2pt) node [right] {};
	\filldraw (1.9,2.5) circle (2pt) node [right] {$\,i$};
	\draw[-] (0,3.2) to[bend right=90] (2,3.2);
	\filldraw (1,2.8) circle (2pt) node [right] {};
	\filldraw (1.7,1.4) circle (2pt) node [right] {\,$j$};
	\draw \firstellip node [label={[xshift=2.1cm, yshift=0.2cm]$H$}] {};
	\draw \thirdellip node [label={[xshift=-2.0cm, yshift=-0.8cm]}] {};
	\end{tikzpicture} 
	\caption{\label{fig:Ex0} Any hypergraph with two isolated nodes is non-stabilizable.}
\end{figure}

\subsection{Stars}
\label{subsec:starcycles}
First recall that, as for any bipartite graph (see Theorem 2 in \cite{MaiMoy16}), graphical matching models on 
trees are always instable. This is true in particular if the matching graph is a ``star'', i.e., a connected graph in which all but one vertices are of degree one. 
The following two results can be seen as generalizations of this fact to hypergraphical models, 

\begin{proposition}
\label{prop:superstar}
If {an} $r$-uniform hypergraph $\mbH=(V,\maH)$ has transversal number $\tau(\mbH)=1$, then it is non-stabilizable. 
\end{proposition}

\begin{proof}
Fix $\Phi$ and $\mu$ in $\textsc{Stab} (\mbH,\Phi)$. Let $T$ be a transversal of cardinality $1$, i.e. $T=\{i_0\}$, where the vertex $i_0$ belongs 
to all hyperedges in $\maH$. Then from Proposition \ref{prop:Ncond3}, we have that $\mu(i_0) < 1/a(\mbH) = 1/r$.  
However, Proposition \ref{prop:Ncond2} implies that $\mu(\{i_0\})>1/r(\mbH)=1/r,$ an absurdity. 
\end{proof}

In other words, any uniform hypergraph whose hyperedges all contain the same node $i_0$ cannot make the corresponding system stable. Moreover, 

\begin{proposition} \label{prop:IpHi}
Suppose that there exists a subset $B \subset V$ in the hypergraph $\mbH=(V,\maH)$ such that: 
\begin{itemize}
\item all hyperedges of $\maH(B)$ contain at least one node of degree 1; 
\item at least one of these nodes of degree 1 lies outside of $B$.
\end{itemize} 
Then $\mbH$ is non-stabilizable. 
\end{proposition}

\begin{proof}
Let $k=|\maH(B)|$, i.e. the number of hyperedges intersecting with $B$. Denote by $H_{1},...,H_{k}$ these intersecting 
hyperedges, and for any $l \in\llbracket 1,k \rrbracket$, by $i_l\in V$, a node of degree one belonging to $H_{l}$. 
Observe that the nodes $i_1,...,i_k$ are not necessarily distinct. On the one hand, for 
any $l\in\llbracket 1,k \rrbracket$ we have that 
\begin{equation*}
X_n(i_l)=A_n(i_l)-M_n(H_{l}).
\end{equation*}
Thus, applying again the right inequality in (\ref{eq:base}) we get that for all $n$, 
\begin{equation*}A_n(B)\geq \sum\limits_{l=1}^k |H_{l}\cap B| M_n(H_l)
={1\over n}\sum\limits_{l=1}^{k}|H_l\cap B|(A_n(i_l)-X_n(i_l)).\end{equation*}
This entails that if $\mu \in \mathscr N_1^{\scriptsize{+}}(\mathbb{H})$, 
\begin{equation*}
\limsup\limits_{n\to\infty}{1 \over n}  \sum\limits_{l=1}^{k}|H_l\cap B|X_n(i_l)\geq\sum\limits_{l=1}^{k}|H_l\cap B|\mu(i_l)-\mu(B) \ge 0.
\end{equation*}
If the above inequality is strong, then the chain $\{W_n\}$ is transient. If the inequality is weak, then as above we can stochastically lower-bound the chain by a zero-drift chain $\{\tilde Y_n\}$, defined by 
 \begin{equation*}
\tilde Y_n =\left(A_n(B)\;-\;\sum\limits_{l=1}^{k}|H_l\cap B|A_n(i_l)\right),\quad n\in\N,
\end{equation*}
which is not identically null from the assumption that at least one of the nodes $i_l$, $l=1,...,k$ is not an element of $B$. This concludes the proof. 
\end{proof}

\begin{ex}\label{ex:chaine}
\rm
Any hypergraph $\mbH=(V,\maH)$ such that there exist two hyperedges $H_1$ and $H_2$ 
with $H_1 \cap H_2 \ne \emptyset$ and two nodes $i_1 \in H_1 \cap \overline{H_2}$, $i_2 \in H_2 \cap \overline{H_1}$ and 
$d\left(i_1\right)=d\left(i_2\right)=1$ is non-stabilizable (see Figure \ref{fig:Ex2}). 
To see this, take $B = H_1 \cap H_2$ in Proposition \ref{prop:IpHi}. 
\begin{figure}[htp]
	\def\firstellip{(1, 2) ellipse [x radius=2cm, y radius=1cm, rotate=180]}
	\def\secondellip{(3.5, 2cm) ellipse [x radius=2cm, y radius=1cm, rotate=180]} 
	\def\thirdellip{(-1.5, 2) ellipse [x radius=2cm, y radius=1cm, rotate=-180]}
	\begin{tikzpicture}
	
	
	\filldraw (-2.5,1.5) circle (2pt) node [right] {$\;i_1$};
	\filldraw (-2.5,2.5) circle (2pt) node [right] {};
	\draw[-] (-4,3.2) to[bend right=90] (-1,3.2);
	\filldraw (-0.3,2.5) circle (2pt) node [right] {};
	\filldraw (-0.3,1.5) circle (2pt) node [right] {};
	\filldraw (1.9,2.5) circle (2pt) node [right] {$\,i_2$};
	\draw[-] (0,3.2) to[bend right=90] (2,3.2);
	\filldraw (1,2.8) circle (2pt) node [right] {};
	\filldraw (1.9,1.25) circle (2pt) node [right] {};
	\draw[-] (1,0.85) to[bend left=90] (3,0.85);
	\draw \firstellip node [label={[xshift=2.1cm, yshift=0.2cm]$H_2$}] {};
	\draw \thirdellip node [label={[xshift=-2.0cm, yshift=-1cm]$H_1$}] {};
	\end{tikzpicture} 
	\caption{\label{fig:Ex2} Two intersecting hyperedges containing each, an isolated node outside of their intersection, make the system instable.}
\end{figure}
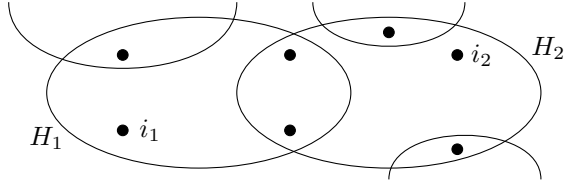

\end{ex}


\subsection{$k$-partite hypergraphs}
We now turn to hypergraphical generalizations of $k$-partite graphs. 

\begin{definition}
\rm An hypergraph $\mathbb{H}=(\maV,\maH)$ is said to be $k$-partite, where $k\ge a(\mbH)$, if there exists a partition $V_1, V_2,\cdots\,, V_k$ of $\maV$ such that every hyperedge in 
$\mbH$ meets each of the $V_i$'s at precisely one vertex, i.e. for any $H\in\maH$ and any $i\le k$, $\left|H\cap V_i\right|=1$. 
We some abuse, we say that a $r$-uniform hypergraph $\mbH$ is $r$-uniform bipartite, if there exists a partition $V_1,V_2$ of $V$ such that for any $H\in\maH$, $|H\cap V_1|=1$ and $|H\cap V_2|=r-1.$ (Notice that such hypergraph cannot be $2$-partite unless it is a graph.) 
\end{definition}

\begin{proposition}
	\label{prop:HbipartiteInstable}
	Any $r$-uniform bipartite hypergraph $\mathbb{H}=(\maV,\maH)$ is non-stabilizable. \end{proposition}
\begin{proof}
	Applying (\ref{eq:base}) successively to $V_1$ and $V_2$ readily implies that for all $n$, 
	\begin{equation*}
	X_n(V_1) =A_n(V_1)-\mathcal{M}_n(\maH)\ge 0 \quad \mbox{ and }\quad
	X_n(V_2) = A_n(V_2)-(r-1)\mathcal{M}_n(\maH)\geq 0,
	\end{equation*} 
	and thus \[X_n(V_1) \geq A_n(V_1)-\displaystyle\frac{1}{r-1}A_n(V_2).\]
	Then, the usual SLLN-based argument implies that the model cannot be stable unless $\mu(V_2) \geq (r-1)\mu(V_1)$. 
	But as $\mu(V_1)+\mu(V_2)=1$ we have that $\mu(V_1)\leq\displaystyle\frac{1}{r}$, 
	hence $\mu \not\in \mathscr N_2(\mbH)$ since $V_1$ is a transversal. 
\end{proof}


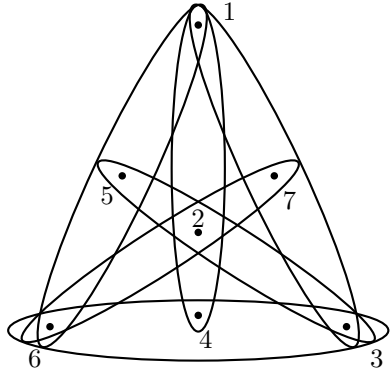
\begin{figure}[htp!]
	\def\firstellip{(4, 3.9) ellipse [x radius=5cm, y radius=0.8cm, rotate=180]}
	\def\secondellip{(2,8) ellipse [x radius=5cm, y radius=0.77cm, rotate=65]} 
	\def\thirdellip{(4, 8.2) ellipse [x radius=4.33cm, y radius=0.7cm, rotate=90]} 
	\def\fourthellip{(6, 8cm) ellipse [x radius=5cm, y
		radius=0.77cm, rotate=-65]}	
	\def\fifthdellip{(3,6) ellipse [x radius=4.33cm, y radius=0.6cm, rotate=33]} 
	\def\sixthdellip{(5,6) ellipse [x radius=4.33cm, y radius=0.6cm, rotate=-33]} 
	
	\begin{tikzpicture}[thick, scale=0.5]
	\filldraw 
	(4,12) circle (2pt) node [label={[xshift=0.41cm, yshift=-0.2cm]$1$}] {};
	\filldraw 
	(4,6.5) circle (2pt) node [label={[xshift=0cm, yshift=-0.19cm]$2$}] {};
	\filldraw 
	(4,4.3) circle (2pt) node [label={[xshift=0.1cm, yshift=-0.7cm]$4$}] {};
	\filldraw 
	(0.1,4) circle (2pt) node [label={[xshift=-0.2cm, yshift=-0.8cm]$6$}] {};
	\filldraw 
	(2,8) circle (2pt) node [label={[xshift=-0.2cm, yshift=-0.64cm]$5$}] {};
	\filldraw 
	(7.9,4) circle (2pt) node [label={[xshift=0.4cm, yshift=-0.8cm]$3$}] {};
	(
	\filldraw 
	(6,8) circle (2pt) node [label={[xshift=0.2cm, yshift=-0.7cm]$7$}] {};
	\draw \firstellip node [label={[xshift=4cm, yshift=-0.5cm]}] {};
	\draw \secondellip node [label={[xshift=4.5cm, yshift=-0.6cm]}] {};
	\draw \thirdellip node [label={[xshift=4.5cm, yshift=-0.5cm]}] {};
	\draw \fourthellip node [label={[xshift=3cm, yshift=-2.5cm]}] {};
	\draw \fifthdellip node [label={[xshift=3cm, yshift=-2.5cm]}] {};
	\draw \sixthdellip node [label={[xshift=3cm, yshift=-2.5cm]}] {};
	\end{tikzpicture} 
	\caption{\label{fig:FanoPlane}The Fano plane minus the hyperedge $\{4,5,7\}$.}
	\end{figure}

\begin{ex}\rm
Set $\mbH$ as the Fano plane, i.e., the smallest projective plane. Namely, $\mbH=(V,\maH)$, where e.g. $V=\llbracket 1,7 \rrbracket$ and 
		$$\maH=\left\{\{1,2,4\},\{1,5,6\},\{1,3,7\},\{2,3,5\},\{4,5,7\},\{4,3,6\},\{6,2,7\}\right\}.$$
If $\mathbb{H}'=(\maV,\maH')$  is the sub-hypergraph defined by $\maH'=\maH\backslash H$, where $H$ is an arbitrary hyperedge of $\maH$. 
	Then it is easily seen that $\mathbb{H}'$ is a 3-uniform bipartite hypergraph with $V_1=H$ and $V_2=\maV\backslash H$. So we deduce from 
	Proposition \ref{prop:HbipartiteInstable} that $\mathbb{H}'$ is non-stabilizable. A Fano plane minus the hyperedge 
	$\{4,5,7\}$ is represented in Figure \ref{fig:FanoPlane}. 
\end{ex}


\begin{proposition}
	\label{prop:Hr-partiteInstable}
	Any $r$-uniform and $k$-partite hypergraph $(\textrm{for }k \ge r)$ is non-stabilizable. 
	\end{proposition}

\begin{proof}
As in the above proof we get that for any $i\ne j$ and any $n$, 
		\begin{equation*}
		X_n(V_j) =A_n(V_j)-\mathcal{M}_n(\maH)\ge 0 \quad \mbox{ and }\quad X_n(V_i) =A_n(V_i)-\mathcal{M}_n(\maH)\ge 0,
		\end{equation*}
		implying that $X_n(V_i)\ge A_n(V_i)-A_n(V_j)$, and in turn, that the model cannot be stable unless $\mu(V_i) \le \mu(V_j)$. 
		By symmetry, this implies that $\mu(V_i)=\mu(V_j)$. As the $V_i$'s are disjoint, we thus have that $\mu(V_i)=1/k$ for all $i$. 
		Thus, as any $V_i$ is a transversal of $\mbH$, $\mu $ is not an element of $\mathscr N_2(\mbH)$. 
		\end{proof}

\begin{definition}
\rm
 An hypergraph $\mbH=(V,\maH)$ satisfies Hall's condition if $|V_2| \geq |V_1|$ for any disjoint subsets $V_2$ and $V_1$ of $V$ satisfying 
	$|H\cap V_2|\geq |H\cap V_1|$ for all hyperedges $H\in\maH$. 
	\end{definition}
It is well known (see \cite{Hall35} for the particular case of graphs, and the general result in \cite{PH95}) that Hall's condition is necessary and sufficient for the existence of a perfect matching on $\mbH$, i.e. a spanning sub-hypergraph of $\mbH$ in which all nodes have degree 1, in the case where the hypergraph is balanced, i.e. it does not contain any odd 
strong cycle. It is intuitively clear that the construction of stable stochastic matching models on hypergraphs is somewhat reminiscent of that of perfect matchings on a growing hypergraphs that replicates the matching hypergraph a large number of times in the long run (in the case of graphs, see the discussion in Section 7 of \cite{MoyPer17}). This connexion has a simple illustration in the next Proposition, which provides a family of probability measures, naturally including the uniform measure on $V$, that cannot stabilize a matching model on the hypergraph $\mbH$ unless the latter satisfies Hall's condition. In what follows we denote for any $\mbH=(V,\maH)$ and any measure $\mu\in\mathscr M(V)$, 
\begin{equation}
\label{eq:defmuminmax}
\mumin=\min\left\{\mu(i)\,:\,i\in V\right\}\quad\mbox{ and }\mumax=\max\left\{\mu(i)\,:\,i\in V\right\}.
\end{equation}
\begin{proposition}
\label{pro:Hall}
	For any hypergraph $\mathbb{H}=(V,\maH)$ that violates Hall's condition, 
	any matching policy $\Phi$ and any $\mu\in \mathscr M(V)$ such that 
	\begin{equation}
	\label{eq:condHall}
	{\mumini \over \mumaxi} > {\left\lfloor {q(\mbH) +1 \over 2}\right\rfloor -1 \over \left\lfloor {q(\mbH) +1 \over 2}\right\rfloor}, 
	\end{equation}
	the model $(\mbH,\Phi,\mu)$ is {instable}. In particular,  $(\mbH,\Phi,\mu_{\textsc{u}})$ is {instable} for $\mu_{\textsc{u}}$ the uniform distribution on $V$. 
\end{proposition}

\begin{proof}
Fix $\mbH$, $\Phi$, and a measure $\mu$ satisfying (\ref{eq:condHall}). We first show that $\mu$ is monotonic with respect to the counting measure on $V$, i.e. 
\begin{equation}
\label{eq:monotonemu}
\forall E,F \subset V,\quad |E| < |F| \Longrightarrow \mu(E) < \mu(F). 
\end{equation}
Let $E$ and $F$ be such that $|E| < |F|$, and let $k = |F|$. Let also $\alpha$ be a bijection from $\llbracket 1,q(\mbH) \rrbracket$ to $V$ such that 
\begin{equation}
\label{eq:defalpha}
\mumin=\mu(\alpha(1)) \le \mu(\alpha(2)) \le ... \le \mu(\alpha(q(\mbH)))=\mumax,
\end{equation}
 in other words 
$\left(\mu(\alpha(1)),\mu(\alpha(2)),...,\mu(\alpha(q(\mbH)))\right)$ is an ordered (in increasing order) version of the family $\left\{\mu(i);\, i \in V\right\}$. 
As $|E|\le k-1$ we clearly have 
\begin{equation}
\label{eq:mad0}
\mu(F) - \mu(E) \ge \sum_{i=1}^k \mu(\alpha(i)) - \sum_{i=q-k+2}^q \mu(\alpha(i)). 
\end{equation}
First, if $k \le \left\lfloor {q(\mbH) +1 \over 2}\right\rfloor$, (\ref{eq:condHall}) entails that $k\mumin >(k-1) \mumax,$ whence 
\begin{equation}
\label{eq:mad1}
\sum_{i=1}^k \mu(\alpha(i)) - \sum_{i=q-k+2}^q \mu(\alpha(i)) \ge k \mumin - (k-1)\mumax >0,
\end{equation}
If $k>\left\lfloor {q(\mbH) +1 \over 2}\right\rfloor$, then the index sets $\llbracket 1,k \rrbracket$ and  $\llbracket q-k+2, q \rrbracket$ intersect precisely 
on $\llbracket q-k+2,k \rrbracket$. 
thus 
\begin{align}
\sum_{i=1}^k \mu(\alpha(i)) - \sum_{i=q-k+2}^q \mu(\alpha(i)) &= \sum_{i=1}^{q-k+1} \mu(\alpha(i)) - \sum_{i=k+1}^q \mu(\alpha(i))\nonumber\\
&\ge  (q-k+1)\mumin - (q-k)\mumax >0, \label{eq:mad2}
\end{align}
where the last inequality follows, as in (\ref{eq:mad1}), from the fact that $q-k+1\le \left\lfloor {q(\mbH) +1 \over 2}\right\rfloor$. 
Gathering (\ref{eq:mad0}) with (\ref{eq:mad1}-\ref{eq:mad2}) concludes the proof of (\ref{eq:monotonemu}) in all cases. 

Now fix $V_2$ and $V_1$ such that $|H\cap V_2|\geq |H\cap V_1|$ for any $H\in\maH$, and $|V_2|<|V_1|$ which from (\ref{eq:monotonemu}), implies that 
$\mu(V_2) < \mu(V_1)$. Then, 
applying again (\ref{eq:base}) to $V_2$ and $V_1$ we get that 
\begin{align*}
\label{eq:compareXY}
X_n(V_2)+X_n(V_1)
&\ge A_n(V_2)+A_n(V_1)-2\sum\limits_{H\in\maH}\left|H \cap V_2\right|\mathcal{M}_n(H)\\
&\ge A_n(V_2)+A_n(V_1)-2A_n(V_2),
\end{align*}
thus, from the usual argument, the model cannot be stable unless $\mu(V_2)\geq \mu(V_1)$, a contradiction.
\end{proof}



\subsection{Cycles}
\label{subsec:cycles}

\begin{definition}
	\rm 
	An $r$-uniform hypergraph $\mbH$ ($r\geq 2$) is called an $\ell$-(Hamiltonian) cycle ($0<\ell<r$), if there exists an ordering $\maV=\left(v_1,v_2,\cdots\,,v_{q(\mbH)}\right)$ of the nodes of 
	$V$ such that:
	\begin{itemize}
		\item {E}very hyperedge of $\maH$ consists of $r$ consecutive nodes modulo $q(\mbH)$; 
		\item Any couple of consecutive hyperedges (in an obvious sense) intersects in exactly $\ell$ vertices. 
	\end{itemize}
\end{definition}

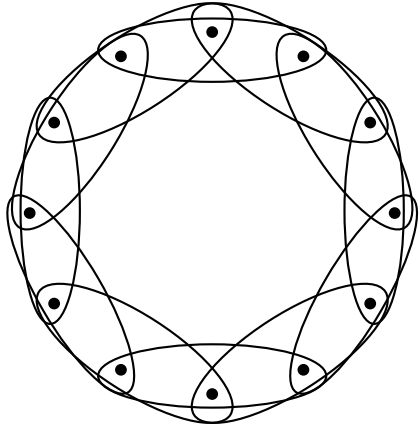
\begin{figure}[htp!]
	\def\tfourthellip{(0, 3.6cm) ellipse [x radius=2.5cm, y radius=0.7cm, rotate=180]} 
	\def\firstellip{(1.7, 3.1) ellipse [x radius=2.5cm, y radius=0.85cm, rotate=-33]}
	\def\secondellip{(2.9, 1.8cm) ellipse [x radius=2.5cm, y radius=0.8cm, rotate=-58]} 
	\def\thirdellip{(3.55, 0.05) ellipse [x radius=2.5cm, y radius=0.65cm, rotate=-90]}
	\def\fourthellip{(3.1, -1.8cm) ellipse [x radius=2.5cm, y radius=0.7cm, rotate=60]} 
	\def\sfirstellip{(1.7, -3.1) ellipse [x radius=2.5cm, y radius=0.85cm, rotate=33]}
	
	\def\ssecondellip{(0, -3.6cm) ellipse [x radius=2.5cm, y radius=0.7cm, rotate=180]} 
	\def\sthirdellip{(-1.7, 3.1) ellipse [x radius=2.5cm, y radius=0.85cm, rotate=33]}
	\def\sfourthellip{(-2.9, 1.8cm) ellipse [x radius=2.5cm, y radius=0.8cm, rotate=58]} 
	\def\tfirstellip{(-3.55, 0.05) ellipse [x radius=2.5cm, y radius=0.65cm, rotate=90]}
	\def\tsecondellip{(-3.1, -1.8cm) ellipse [x radius=2.5cm, y radius=0.7cm, rotate=-60]} 
	\def\tthirdellip{(-1.7, -3.1) ellipse [x radius=2.5cm, y radius=0.85cm, rotate=-33]}

	\begin{tikzpicture}[thick, scale=0.6]
	\filldraw (-0*360/12:4) circle (3pt) node [above]{$$};
	\filldraw (-1*360/12:4) circle (3pt) node [above]{$$};
	\filldraw (-2*360/12:4) circle (3pt) node [above]{$$};
	\filldraw (-3*360/12:4) circle (3pt) node [above]{$$};
	\filldraw (-4*360/12:4) circle (3pt) node [above]{$$};
	\filldraw (-5*360/12:4) circle (3pt) node [above]{$$};
	\filldraw (-6*360/12:4) circle (3pt) node [above]{$$};
	\filldraw (-7*360/12:4) circle (3pt) node [above]{$$};
	\filldraw (-8*360/12:4) circle (3pt) node [above]{$$};
	\filldraw (-9*360/12:4) circle (3pt) node [above]{$$};
	\filldraw (-10*360/12:4) circle (3pt) node [above]{$$};
	\filldraw (-11*360/12:4) circle (3pt) node [above]{$$};
	\draw \firstellip node [label={[xshift=1.5cm, yshift=0.8cm]$$}] {};
	\draw \secondellip node [label={[xshift=2.2cm, yshift=0cm]$$}] {};
	\draw \thirdellip node [label={[xshift=2.0cm, yshift=-0.8cm]$$}] {};
	\draw \fourthellip node [label={[xshift=2.1cm, yshift=-0.6cm]$$}] {};
	\draw \sfirstellip node [label={[xshift=2.1cm, yshift=-0.8cm]$$}] {};
	\draw \ssecondellip node [label={[xshift=0cm, yshift=-2cm]$$}] {};
	\draw \sthirdellip node [label={[xshift=-0.3cm, yshift=1.2cm]$$}] {};
	\draw \sfourthellip node [label={[xshift=-2.2cm, yshift=0.4cm]$$}] {};
	\draw \tfirstellip node [label={[xshift=-1.6cm, yshift=-0.4cm]$$}] {};
	\draw \tsecondellip node [label={[xshift=-1.5cm, yshift=-0.9cm]$$}] {};
	\draw \tthirdellip node [label={[xshift=-1.0cm, yshift=-1.8cm]$$}] {};
	\draw \tfourthellip node [label={[xshift=0cm, yshift=1.1cm]$$}] {};


	\end{tikzpicture}
	\caption{A $3$-uniform $2$-cycle of order $12$. \label{fig:cycle}}
\end{figure}

\begin{proposition}
	\label{prop:cycles}
	Any $r$-uniform $\ell$-cycle of order $q$ such that $r$ divides $q$, is non-stabilizable. 
\end{proposition}
\begin{proof}
	The partition $V_1,V_2,\cdots,V_r$ of $\maV$ defined by 
	$$V_i=\left\lbrace v_{i+(j-1)r}\;;\,j\in \llbracket 1,q/r\rrbracket \right\rbrace,$$
	satisfies Proposition \ref{prop:Hr-partiteInstable} for $k\equiv r$. 
	 \end{proof}

\noindent Figure \ref{fig:cycle} shows a $3$-uniform $2$-cycle of order $12$.


\section{Stable systems}
\label{sec:stable}
We show hereafter that stable matching models on hypergraphs do exist. We provide two examples of hypergraphs on 
which a stable stochastic matching model can be defined: complete $3$-uniform hypergraphs, and sub-hypergraphs of the latter where several hyperedges are erased. 

\subsection{Complete $3$-uniform hypergraphs}
\label{subsec:complete}
We first consider the case of a complete $3$-uniform hypergraph $\mbH$, an example of which for $q(\mbH)=4$ is represented in Figure \ref{fig:complete}. 
We show that in this case, the necessary condition given in Proposition 
\ref{prop:Ncond3} is also sufficient, 

\begin{figure}
	\def\firstellip{(0.15, 2.5) ellipse [x radius=4cm, y radius=0.8cm, rotate=180]}
	\def\secondellip{(0.15, 1.6cm) ellipse [x radius=4cm, y radius=0.8cm, rotate=180]} 
	\def\thirdellip{(0.8, 1.7) ellipse [x radius=3cm, y radius=1.3cm, rotate=90]} 
	\def\fourthellip{(-0.8, 1.7cm) ellipse [x radius=3cm, y
		radius=1.3cm, rotate=-90]}	
	\begin{tikzpicture}
		\filldraw 
	(1.8,2) circle (2pt) node [left] {$\;4$};
	\filldraw 
	(0,2.8) circle (2pt) node [right] {$1$};
	\filldraw 
	(-1.8,2) circle (2pt) node [right] {$\;3$};
	\filldraw 
	(0,1.2) circle (2pt) node [right] {$2$};
	\draw \firstellip node [label={[xshift=4.5cm, yshift=-0.5cm]$H_3$}] {};
	\draw \secondellip node [label={[xshift=4.5cm, yshift=-0.6cm]$H_4$}] {};
	\draw \thirdellip node [label={[xshift=-3cm, yshift=-2.5cm]$H_1$}] {};
	\draw \fourthellip node [label={[xshift=3cm, yshift=-2.5cm]$H_2$}] {};
	\end{tikzpicture} 
	\caption{\label{fig:complete} Complete $3$-uniform hypergraph of order $4$.}
\end{figure}
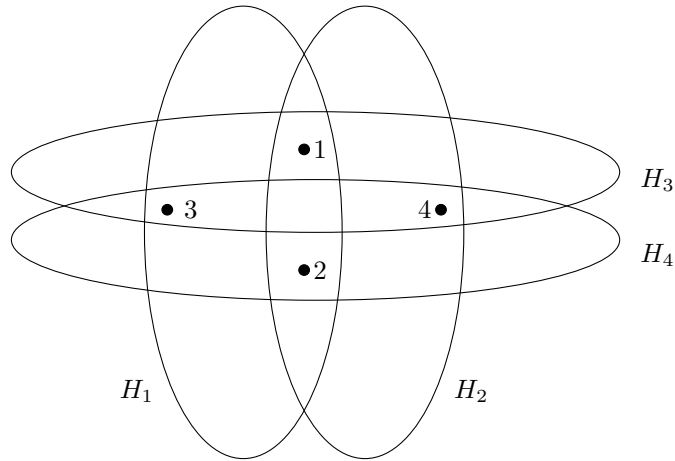

\begin{theorem}
\label{thm:stable3unif}
Let $\mbH$ be a complete $3$-uniform hypergraph of order $q(\mbH)\ge 4$. 
Then, for any admissible policy $\Phi$ we have that 
\[\textsc{Stab}(\mbH,\Phi) = \mathscr N_3^{\scriptsize{-}}(\mbH),\]
that is, the model $(\mbH,\Phi,\mu)$ is stable if and only if $\mu(i) < 1/3$ for any $i\in V$. 
\end{theorem}


\begin{proof}
	Necessity of the condition being shown in Proposition \ref{prop:Ncond3}, only the sufficiency remains to be proven. 
	Suppose that $\mu(i)<1/3$ for any $i\in V$, and fix $\alpha$ such that $\max_{i\in V} \mu(i) < \alpha <1/3$. 
	Define the planar Markov chain $\suite{U^{\alpha}_n}$ having the following transitions on $\N^2$, 
	\[\left\{\begin{array}{llll}
	\mbox{First axis:} \quad  &P^\alpha_{(x,0),(x+1,0)} &= \alpha,&\,x\in\N^+,\\
	&P^\alpha_{(x,0),(x,1)} &= 1-\alpha,&\,x\in\N^+,\\
	\mbox{Second axis:}  \quad  &P^\alpha_{(0,y),(0,y+1)} &= \alpha,&\,y\in\N^+,\\
	&P^\alpha_{(0,y),(1,y)} &= 1-\alpha,&\,y\in\N^+,\\
	\mbox{Interior:}  \quad & P^\alpha_{(x,y),(x+1,y)} &= \alpha,&\,x,y\in\N^+,\\
	&P^\alpha_{(x,y),(x,y+1)} &= \alpha,&\,x,y\in\N^+,\\
	&P^\alpha_{(x,y),(x-1,y-1)} &= 1- 2\alpha,&\,x,y\in\N^+,
	\end{array}\right.\]
	and arbitrary transitions from $(0,0)$ to any element of $\N^2$. (These transitions are represented in Figure \ref{Fig:transU}.) 
	\begin{figure}[htb]
		\begin{center}
			\begin{tikzpicture}
			\draw[->] (0,0) -- (5,0) ;
			\draw[->] (0,0) -- (0,5);
			\fill (2,0) circle (2pt);
			\draw[->, thick] (2,0) -- (2.5,0) node [above right]{\scriptsize{$\alpha$}};
			\draw[->, thick] (2,0) -- (2,0.5) node [above]{\scriptsize{$1-\alpha$}};
			%
			\fill (3,3) circle (2pt);
			\draw[->, thick] (3,3) -- (3.5,3) node [right]{\scriptsize{$\alpha$}};
			\draw[->, thick] (3,3) -- (3,3.5) node [above]{\scriptsize{$\alpha$}};
			\draw[->, thick] (3,3) -- (2.5,2.5) node [below]{\scriptsize{$1-2\alpha$}};
			%
			\fill (0,2) circle (2pt);
			\draw[->, thick] (0,2) -- (0.5,2) node [right]{\scriptsize{$1-\alpha$}};
			\draw[->, thick] (0,2) -- (0,2.5) node [right]{\scriptsize{$\alpha$}};
			%
			%
			%
			
			\end{tikzpicture}
			\caption{\label{Fig:transU} Auxiliary Markov chain of the complete 3-uniform hypergraph.}
		\end{center}
		
	\end{figure}
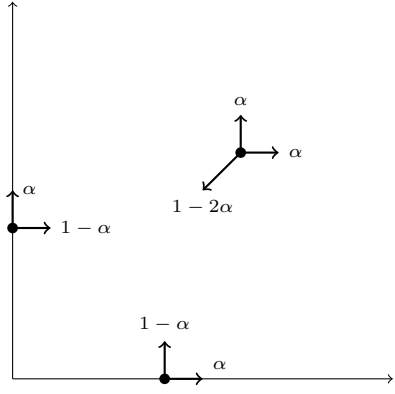
	
	Denote by $\Delta=(\Delta_x,\Delta_y)$, $\Delta'=(\Delta'_x,\Delta'_y)$ and $\Delta''=(\Delta''_x,\Delta''_y)$, the mean (horizontal and vertical) 
	drifts of the chain $\{U^\alpha_n\}$, respectively on the interior, on the first and on the second axis, in a way that 
	\[\left\{\begin{array}{lll}
	\mbox{First axis:} \quad  &\Delta'_x=\alpha, \quad &\Delta'_y=1-\alpha;\\
	\mbox{Second axis:}  \quad  &\Delta''_x=1-\alpha, \quad &\Delta''_y=\alpha;\\
	\mbox{Interior:}  \quad & \Delta_x=3\alpha-1, \quad &\Delta_y=3\alpha-1.
	\end{array}\right.\]
	Thus, $\Delta_x <0$ and $\Delta_y<0$. Also, we have that 
	\[\Delta_x\Delta'_y - \Delta_y\Delta'_x = \Delta_x\Delta''_y - \Delta_y\Delta''_x = (3\alpha-1)(1-2\alpha) <0,\]
	so we can apply Theorem 3.3.1, part (a) of \cite{FMM95}, to claim that the Markov chain $\{U^\alpha_n\}$ is positive recurrent. 
	Specifically, it can be checked that, setting $u={1-3\alpha\over 2}>0$, for any $w$ such that ${{3\alpha-1}} < w <{(3\alpha - 1)\alpha \over 1-\alpha}<0$ we have that 
	\begin{equation}
	\label{eq:deltas}
	\begin{cases}
	2u \Delta_x + w \Delta_y &< 0,\\
	2u \Delta_y + w \Delta_x &< 0,\\
	2u \Delta'_x + w \Delta'_y &< 0,\\
	2u \Delta''_y + w \Delta_x &< 0.
	\end{cases}
	\end{equation}
	Second, as $4u^2> w^2$ the quadratic form $Q:(x,y) \mapsto ux^2 +uy^2+wxy$ is positive definite. 
	Then, in view of Lemma 3.3.3 in \cite{FMM95}, it follows from (\ref{eq:deltas}) that, defining the mapping 
	\[L^\alpha : 
	\begin{cases}
	\N^2 &\longrightarrow \R_+ \\
	(x,y) &\longmapsto  \sqrt{Q(x,y)}=\sqrt{ux^2 +uy^2+wxy}, 
	\end{cases}\]
	we have that for some compact set $\mathcal K^\alpha\subset \N^2$, for any $(x,y) \in \left(\mathcal K^\alpha\right)^c$, 
	\begin{equation}
	\label{eq:lyapua}
	\esp{L^\alpha\left(U^\alpha_{n+1}\right) - L^\alpha(U^\alpha_n) \mid U^\alpha_n = (x,y)} <0. 
	\end{equation}
	
	Now, as $\mbH$ is complete $3$-uniform the states of the Markov chain $\suite{X_n}$ have at most two non-zero coordinates, 
	in other words its state space is 
	\[\mathcal E = \Bigl\{\grx=(x_1,...,x_q) \in\N^q \,:\, x_ix_jx_k=0 \mbox{ for any distinct }i,j,k \in \llbracket 1,q \rrbracket\Bigl\}.\]
	Define the mapping 
	\[L: 
	\left\{\begin{array}{ll}
	\mathcal E &\longrightarrow \R_+\\
	\grx &\longmapsto \left\{\begin{array}{ll}
	0      &\mbox{ if } \grx=\mathbf 0,\\ 
	L^\alpha((x,0))    &\mbox{ if $\grx=x.\gre_i$, for some $x>0$, $i\in V$},\\
	L^\alpha((x,y))  &\mbox{ if $\grx=x.\gre_i+y.\gre_j$, for some $x,y>0$, $i\ne j$,}  
	\end{array}\right.
	\end{array}\right.\]    
	where the above definition is unambiguous due to the fact that $L^\alpha$ is a symmetric form on $\N^2$. 
	Also define the compact set 
	\[\mathcal K = \left\{\grx :=x.\gre_i+y.\gre_j\in \mathcal E\,:\, (x,y) \in\mathcal K^\alpha\right\}.\]
	Then, first, if $\grx\in\mathcal K^c\cap \mathcal E$ is such that $\grx=x.\gre_i+ y.\gre_j$ for some $x,y>0$ and $i,j \in V$, $i\ne j$, we get that 
	\begin{multline*}
	\esp{L\left(X_{n+1}\right) - L(X_n) \mid X_n = \grx} \\
	\shoveleft{=(1-\mu(i)-\mu(j))\left(L\left(\grx-\gre_i-\gre_j\right)-L(\grx)\right)}\\
	\shoveright{+ \mu(i) \left(L\left(\grx+\gre_i\right)-L(\grx)\right) +  \mu(j) \left(L\left(\grx+\gre_j\right) - L(\grx)\right)}\\
	\shoveleft{=(1-\mu(i)-\mu(j))\left(L^\alpha\left(x-1,y-1\right)-L^\alpha(x,y)\right) }\\
	\shoveright{+ \mu(i) \left(L^\alpha\left(x+1,y\right)-L^\alpha(x,y)\right)+  \mu(j) \left(L^\alpha\left(x,y+1\right) - L^\alpha(x,y)\right)}\\
	\shoveleft{{< (1-2\alpha)}\left(L^\alpha\left(x-1,y-1\right)-L^\alpha(x,y)\right) }\\
	\shoveright{+ \alpha \left(L^\alpha\left(x+1,y\right)-L^\alpha(x,y)\right)+  \alpha \left(L^\alpha\left(x,y+1\right) - L^\alpha(x,y)\right)}\\
	= \esp{L^\alpha\left(U^\alpha_{n+1}\right) - L^\alpha(U^\alpha_n) \mid U^\alpha_n = (x,y)},
	\end{multline*}
	where, in the inequality above, we used the facts that $L^\alpha$ is non-decreasing in its first and second variables, and such that
	$L^\alpha\left(x-1,y-1\right){<}\, L^\alpha(x,y)$. 
	Likewise, if $\grx\in \mathcal K^c \cap \mathcal E$ is such that $\grx=x.\gre_i$ for some $x>0$ and $i\in V$, we have that 
	\begin{multline*}
	\esp{L\left(X_{n+1}\right) - L(X_n) \mid X_n = \grx} \\
	\begin{aligned}
	&=\sum_{j\ne i}\mu(j)\left(L\left(\grx+\gre_j\right)-L(\grx)\right)+ \mu(i) \left(L\left(\grx+\gre_i\right)-L(\grx)\right)\\
	&=(1-\mu(i))\left(L^\alpha\left(x,1\right)-L^\alpha(x,0)\right)+ \mu(i) \left(L^\alpha\left(x+1,0\right)-L^\alpha(x,0)\right)\\
	&{< (1-\alpha)}\left(L^\alpha\left(x,1\right)-L^\alpha(x,0)\right)+ \alpha \left(L^\alpha\left(x+1,0\right)-L^\alpha(x,0)\right)\\
	&= \esp{ L^\alpha\left(U^\alpha_{n+1}\right) - L^\alpha(U^\alpha_n) \mid U^\alpha_n = (x,0)},
	\end{aligned}
	\end{multline*}
	remarking that $L^\alpha(x,1) { < }\, L^\alpha(x,0)$. Recalling that $X_n=[W_n]$ for all $n$, 
	using (\ref{eq:lyapua}) in both cases, we conclude using the Lyapunov-Foster Theorem (see e.g. \cite{Bre99}, \S  5.1) that the chain $\suite{W_n}$ is positive recurrent. 
\end{proof}


\subsection{{Incomplete $3$-uniform hypergraphs}}
As is shown in  Theorem \ref{thm:stable3unif}, complete $3$-uniform hypergraphs are stabilizable for a large class of measures. 
We show hereafter that incomplete hypergraphs can also be stabilizable, 

\begin{theorem}
	\label{thm:suff3unifincomplet} 
	Let $\mathbb H=(\maV,\maH)$ be a complete $3$-uniform hypergraph of {order} $q \ge 5$, and let 
	$\mbH'=(\maV,\maH')$ be 
	the $(3\textrm{-uniform})$ sub-hypergraph of $\mbH$ obtained by setting $\maH'=\maH\backslash \maJ$, where  $\maJ$ is a subset of $\maH$ containing disjoint hyperedges. Let $J$ be the union of the elements of $\maJ$.
	 Then the model $(\mbH',\textsc{ml},\mu)$ is stable for any 
	$\mu$ in 
	\[\mathscr S(\mbH')=\biggl\{\mu\in\mathscr M(V):\left(\max\limits_{i\in J} \lambda_i(\mu)\vee \max\limits_{i\in \bar J} \nu_i(\mu)\right)<0\biggl\}\,\cap\,\mathscr N_2(\mathbb H')\, \cap\mathscr N^{\scriptsize{-}}_3(\mathbb H'),\]
	{where the $\lambda_i(\mu): i\in J$ and $\nu_{i}(\mu): i\in \bar J$ are defined respectively by (\ref{eq:Gdeflambdai}) and (\ref{eq:InCnu_i})}. 
\end{theorem}

\begin{proof}
Fix $\Phi=\textsc{ml}$, and let $\mu\in {\mathscr S(\mbH')}$. For such $\mbH'$ the study of $\suite{Y_n}$ does not boil down to that of a planar Markov chain. 
Instead, we study the embedded chain $\suiten{Y_n}=\suiten{X_{4n}}$, 
 and consider the following quadratic Lyapunov function,  
\[L:\left\{\begin{array}{ll}
\N^q &\longrightarrow \R_+;\\
x &\longmapsto \sum_{i=1}^{q} (x_i)^2. 
\end{array}\right.\]
Fix $n\in\N$. We have the following alternatives given the value of the embedded chain $\{Y_n\}$ at time $n$,

(i) First, for any $i\in \overline{J}$, and any integer $x_i\geq{2}$, the chain $\suite{Y_n}$ makes the following transitions from state $Y_n=x_i.\gre_i$, 
\scriptsize{
	\begin{equation}
	\label{eq:GtransY1}
	Y_{n+1}=\left\{\begin{array}{ll}
	Y_n+4\gre_i &\mbox{ w.p. }\mu(i)^4;\\
	Y_n+3\gre_i+\gre_j &\mbox{ w.p. }4\mu(i)^3\mu(j);\\
	Y_n+2\gre_i+2\gre_j &\mbox{ w.p. }6\mu(i)^2\mu(j)^2;\\
	Y_n+\gre_i+3\gre_j &\mbox{ w.p. }4\mu(i)\mu(j)^3;\\
	Y_n+4\gre_j &\mbox{ w.p. }\mu(j)^4;\\
	Y_n+\gre_i &\mbox{ w.p. }12\mu(i)^2\mu(j)\mu(k);\\
	Y_n+\gre_j &\mbox{ w.p. }12\mu(i)\mu(j)^2\mu(k);\\
	Y_n-2\gre_i &\mbox{ w.p. }10\mu(j)^2\mu(k)\mu(\ell)\\
	& \mbox{ (\scriptsize{the input has 2 $j$, 1 $k$ and 1 $\ell$, but does not end in $jj$});}\\
	Y_n-\gre_i+2\gre_j &\mbox{ w.p. }2\mu(k)\mu(\ell)\mu(j)^2\\ 
	& \mbox{ (\scriptsize{the input has 2 $j$, 1 $k$ and 1 $\ell$, but  ending in $jj$});}\\
	Y_n-2\gre_i &\mbox{ w.p. }6\mu(j)^2\mu(k)^2;\\
	Y_n-\gre_i+2\gre_j &\mbox{ w.p. }4\mu(j)^3\mu(k);\\
	Y_n+\gre_j &\mbox{ w.p. }24\mu(i)\mu(j)\mu(k)\mu(\ell);\\
	Y_n-2\gre_i &\mbox{ w.p. }24\mu(j)\mu(k)\mu(\ell)\mu(m).\\
	\end{array}\right.
	\end{equation}}
\normalsize{From this, we deduce using simple algebra that}
\begin{equation*}
\Delta_i:=\esp{L\left(Y_{n+1}\right) - L\left(Y_n\right) | Y_n = x_i.\gre_i}=\lambda_i(\mu)x_i+\beta_i(\mu), 
\end{equation*}
for some bounded $\beta_i(\mu)$, and for
\begin{multline}
\label{eq:Gdeflambdai}
\lambda_i(\mu)=\;8\mu^4(i) + 24\mu^3(i)\sum_{j\ne i} \mu(j) + 24\mu^2(i)\sum_{j\ne i} \mu^2(j)\\ + 8\mu(i)\sum_{j\ne i} \mu^3(j)
+24\mu^2(i) \sum_{\substack{j,k\ne i}} \mu(j)\mu(k)
-44\sum_{\substack{j,k,\ell\ne i}} \mu^2(j)\mu(k)\mu(\ell)\\-24\sum_{\substack{j,k\ne i}} \mu^2(j)\mu^2(k)
-8\sum_{\substack{j,k\ne i}}\mu(j)\mu^3(k)-96\sum_{\substack{j,k,\ell,m\ne i}} \mu(j)\mu(k)\mu(\ell)\mu(m). 
\end{multline}
Consequently, as the above is negative there exists $a_1^{*}$ such that $\Delta_i <0$ whenever $x_i \ge a_1^{*}$.

\medskip

(ii) For any $i\in J$, and any integer $x_{i}\geq{2}$, the transitions of $\suite{Y_n}$ from the state $x_{i}.\gre_i$ can be retrieved in a similar fashion to 
(\ref{eq:GtransY1}). It follows that 
\begin{equation*}
\Delta'_{i}=\esp{L\left(Y_{n+1}\right) - L\left(Y_n\right) | Y_n = x_{i}.\gre_{i}}=\nu_{i}(\mu)x_{i}+\beta'_{i}(\mu), 
\end{equation*}
for some bounded $\beta'_{i}(\mu)$, and setting $H=\{i,j,k\}$ as the only element of $\maJ$ such that $i\in H$, we obtain that 
\scriptsize{
\begin{multline}
\label{eq:InCnu_i}
\nu_i(\mu)=8\mu^4(i) + 24\mu^3(i)\sum_{\ell\in\maV\backslash\{i\}} \mu(\ell) + 24\mu^2(i)\sum_{\ell\in\maV\backslash \{i\}} \mu^2(\ell) + 8\mu(i)\sum_{\ell\in\maV\backslash\{i\}} \mu^3(\ell)\\
-8 \sum_{\substack{\ell\in \overline{H}}} \mu(j) \mu^3(\ell)
-4 \sum_{\substack{\ell\in\overline{H}:\\\textrm{ends with }kk}}  \mu(j)\mu^2(k)\mu(\ell)
-20\sum_{\substack{\ell\in\overline{H}:\\\textrm{otherwise}}} \mu(j)\mu^2(k)  \mu(\ell)\qquad\qquad\qquad\;\;\\
-48 \mu(j)\mu(k) \sum_{\substack{\ell\in\overline{H}}} \mu^2(\ell)
-4 \sum_{\substack{\ell\in\overline{H}:\\\textrm{ends with }\ell\ell}} \mu(j)\mu^2(\ell)\mu(m)
-40 \sum_{\substack{\ell\in\overline{H}:\\\textrm{otherwise}}} \mu(j)\mu^2(\ell)\mu(m)\qquad\quad\;\;\\
+ 48 \mu^2(i)\mu(j)\mu(k)-8 \mu(j)\mu(k)\sum_{\substack{\ell,m\in\overline{H}:\\\textrm{ends with }jk}} \mu(\ell)\mu(m)\
-80 \mu(j)\mu(k)\sum_{\substack{\ell,m\in\overline{H}:\\\textrm{otherwise}}} \mu(\ell)\mu(m)\qquad\;\\
-96 \sum_{\substack{\ell,m\in\overline{H}}} \mu(j)\mu(\ell)\mu(m)\mu(p)
+24 \mu^2(i)\sum_{\substack{\ell\in\overline{H}}} \mu(j)\mu(\ell)
+24 \mu^2(i)\sum_{\substack{\ell,m\in\overline{H}}} \mu(\ell)\mu(m)\qquad\qquad\;\;\;\\
-24\sum_{\substack{\ell\in\overline{H}}} \mu^2(j)\mu^2(\ell)
-4\sum_{\substack{\ell\in\overline{H}:\\\textrm{ends with }jj}} \mu^2(j)\mu(\ell)\mu(m)
-40\sum_{\substack{\ell\in\overline{H}:\\\textrm{otherwise}}} \mu^2(j)\mu(\ell)\mu(m)\qquad\qquad\;\;\;\;\\
-8\sum_{\substack{\ell\in\overline{H}}} \mu^3(j)\mu(\ell)
+24\mu(i)\sum_{\substack{j,k\in H}}\mu(j)\mu^2(k)
-8\sum_{\substack{\ell,m\in\overline{H}}} \mu^3(\ell)\mu(m)
-24\sum_{\substack{\ell,m\in\overline{H}}} \mu^2(\ell)\mu^2(m)\;\;\;\\
-4\sum_{\substack{\ell,m,p\in\overline{H}:\\ \textrm{ends with }\ell\ell}} \mu^2(\ell)\mu(m)\mu(p)-40\sum_{\substack{\ell,m,p\in\overline{H}:\\\textrm{otherwise}}} \mu^2(\ell)\mu(m)\mu(p)
-96\sum_{\substack{\ell,m,p,s\in\overline{H}}} \mu(\ell)\mu(m)\mu(p)\mu(s).
\end{multline}}
\normalsize{Thus, there exists $a_{2}^{*}$ such that $\Delta'_{i} <0$ whenever $x_{i} \ge a_{2}^{*}$.}

\medskip
(iii) For any $i\ne j$ such that $\{i,j\}$ is not included in an hyperedge of the family $\mathcal J$, 
for any integers $x_{i},x_{j}>0$, we obtain that 
\begin{equation*}
\Delta_{ij}:=\esp{L\left(X_{n+1}\right) - L\left(X_n\right) | X_n = x_{i}.\gre_i+x_{j}.\gre_j}=\lambda_{ij}(\mu)x_{i}+\lambda_{ji}(\mu)x_{j}+\beta_{ij}(\mu), 
\end{equation*}
for a bounded $\beta_{ij}(\mu)$, and for 
\begin{equation}
\label{eq:Gdeflambdaij,ji}
\lambda_{ij}(\mu)=2\Big(\mu(i)-\sum\limits_{\ell \in \maV\backslash \{i,j\}}\mu(\ell)\Big)\textrm{ and }\lambda_{ji}(\mu)=2\Big(\mu(j)-\sum\limits_{\ell \in \maV\backslash \{i,j\}}\mu(\ell)\Big).
\end{equation}
Now observe that $\maV\backslash \{i,j\}\in\mathcal{T}(\mbH)$, so $\sum\limits_{\ell \in \maV\backslash \{i,j\}}\mu(\ell)>\displaystyle\frac{1}{3}$ and then $\lambda_{ij}(\mu)<0$ and $\lambda_{ji}(\mu)<0$.
Thus there exists $a_{3}^{*}$ such that $\Delta_{ij}<0$ whenever ${x_{i} \vee x_{j}} \ge a_{3}^{*}$. 

\medskip

(iv) For any $i,j$ such that $i\ne j$ and $\{i,j\}\subset H$ for some $H\in \mathcal J$, for any integers $x_{i},x_{j}>0,$ we obtain that
\begin{equation*}
\Delta'_{ij} :=\esp{L\left(X_{n+1}\right) - L\left(X_n\right) | X_n = x_{i}.\gre_{i}+x_{j}.\gre_{j}}
       = \nu_{ij}(\mu)x_i + \nu_{ji}(\mu)x_j +\beta'_{ij}(\mu)
\end{equation*}
for a bounded $\beta'_{ij}(\mu)$ and 
\begin{equation}
\label{eq:Gdefnuij,ji}
\nu_{ij}(\mu)=2\Big(\mu(i)-\sum\limits_{\ell \in \overline{H}}\mu(\ell)\Big)\textrm{ and }\nu_{ji}(\mu)=2\Big(\mu(j)-\sum\limits_{\ell \in \overline{H}}\mu(\ell)\Big).
\end{equation}

But $\overline{H}\in\mathcal{T}(\mbH)$, so $\sum\limits_{\ell \in\overline{H}}\mu(\ell)>\displaystyle\frac{1}{3}$ and thus $\nu_{ij}(\mu)<0$ and $\nu_{ji}(\mu)<0$.\\
\normalsize{Again, there exists $a_{4}^{*}$ such that $\Delta'_{ij}<0$ whenever ${x_{i} \vee x_{j}} \ge a_{4}^{*}$.}

\medskip

	(v) We finally consider the case where $X_n=x_i.\gre_{i}+x_j.\gre_{j}+x_k.\gre_{k}$ for $H=\{i,j,k\}$, for some $H\in\maJ$, 
	and integers $x_i,x_j$ and $x_k$ such that $x_i,x_j \geq x_k>0$. \\
	\begin{align*}
	\Delta_H &:=\esp{L\left(X_{n+1}\right) - L\left(X_n\right) | X_n = x_{i}.\gre_i+x_{j}.\gre_j+x_{k}.\gre_k}\\
	               & =\alpha_{i}(\mu)x_{i}+\alpha_{j}(\mu)x_{j}+\alpha_{k}(\mu)x_k+\beta_H(\mu), 
	\end{align*}
	for a bounded $\beta_H(\mu)$, and for 
	\begin{equation}
	\label{eq:Gdeflaphai,j,k}
	\alpha_{i}(\mu)=2\Big(\mu(i)-\sum\limits_{\ell \in \overline{H}}\mu(\ell)\Big),\;\alpha_{j}(\mu)=2\Big(\mu(j)-\sum\limits_{\ell \in \overline{H}}\mu(\ell)\Big)\textrm{ and }
	\alpha_{k}(\mu)=2\mu(k).
	\end{equation}
	Now observe that $\overline{H}\in\mathcal{T}(\mbH)$, so $\alpha_{i}(\mu)<0$ and $\alpha_{j}(\mu)<0$.  
	From this, we deduce as above the existence of an integer $a_5^{*}$ such that $\Delta_H<0$ whenever 
	$x_i\vee x_j \ge a_5^{*}$.

%
\medskip

To conclude, if we let $\mathcal{K}$ be the finite set
\[{\mathcal K=\left\lbrace x\in\maE\;:\;x_i\le \max\Big(a_1^*,...,a_5^*,2\Big);\; i\in V \right\rbrace,}\]
then if follows from the above arguments that for any $x\in\maE\cap\mathcal{K}^c$ 
and any $n\in\N$, 
$$\esp{L\left(Y_{n+1}\right) - L\left(Y_{n}\right) | Y_n=x}<0.$$
We deduce from Lyapunov-Foster Theorem (see \cite{Bre99}, \S  5.1) that the chain $\suiten{Y_n}$ is positive recurrent. 
This is the case in turn for the chain $\suiten{X_n}$.
\end{proof}

\begin{remark}
\rm 
Observe that the only incomplete (in the sense of Theorem \ref{thm:suff3unifincomplet}) $3$-uniform hypergraph of order 
4 would be obtained from the complete one by deleting only one vertex. However, as easily seen the transversal number of the resulting hypergraph is 1, so the latter is non-stabilizable from Proposition \ref{prop:superstar}.
 \end{remark}


In the following examples we show how stability can be shown for various $3$-uniform incomplete hypergraphs using 
Theorem \ref{thm:suff3unifincomplet}, 



\begin{corollary}
Consider an incomplete $3$-uniform hypergraph $\mathbb{H}'$ satisfying the assumptions of Theorem \ref{thm:suff3unifincomplet}. 
Recall (\ref{eq:defmuminmax}), and define the sets 
\[\mathscr A(\mbH'):=\biggl\{\mu\in\mathscr M(V)\,:\, \displaystyle\frac{\mumaxi}{\mumini}<\left(\frac{2q^4-9q^3+12q^2-13q+12}{6q^2+10q+24}\right)^{1/4}\biggl\}\,\]
and 
\begin{equation*}
\mathscr S_1(\mbH'):=\mathscr A(\mbH')\,\cap\,\mathscr N_2(\mathbb H')\, \cap\mathscr N^{\scriptsize{-}}_3(\mathbb H').	\end{equation*}
Then the model 
$(\mathbb{H}',\textsc{ml},\mu)$ is stable for any  $\mu\in\mathscr S_1(\mbH').$	
\end{corollary}
\begin{proof}
Recalling (\ref{eq:Gdeflambdai}) and (\ref{eq:InCnu_i}), a simple algebra shows that 
\[\mathscr A(\mbH') \subset \biggl\{\mu\in\mathscr M(V):\left(\max\limits_{i\in J} \lambda_i(\mu)\vee \max\limits_{i\in \bar J} \nu_i(\mu)\right)<0\biggl\},\]
thus $\mathscr S_1(\mbH')\subset \mathscr S(\mbH')$. 
\end{proof}

\begin{ex}
	\label{ex:UnifDisIncomplet}
	\rm
	Observe that for any such $\mbH'=(V,\maH')$ satisfying the assumptions of Theorem \ref{thm:suff3unifincomplet}, 
	the model $(\mathbb{H}',\textsc{ml},\mu_{\textsc{u}})$ is stable for $\mu_{\textsc{u}}$ the uniform distribution on $V$. 
	Indeed, we have $\mu_{\textsc{u}}\in\mathscr S_1(\mbH')$.  
	To see this, first observe that $\frac{2q^4-9q^3+12q^2-13q+12}{6q^2+10q+24}>1$. Moreover, it is immediate that 
	$\mu_{\textsc{u}}\in\mathscr N_3^-(\mbH')$. 
	It remains to show that $\mu_{\textsc{u}}\in\mathscr N_2(\mbH')$. 
	We proceed in three steps. First, for $q=5$ the only incomplete $3$-uniform graph in the sense of Theorem \ref{thm:suff3unifincomplet} is the complete hypergraph on $\llbracket 1,5 \rrbracket$ minus one vertex, say $\{1,2,3\}$. 
	It is then easily seen that $\{4,5\}$ is the only minimal transversal of $\mbH'$. So $\tau(\mbH')=2$, in a way that for 
	all $T'\in\mathcal T(\mbH')$, $\mu_{\textsc{u}}(T')\ge 2/5> 1/3,$ 
	showing that $\mu_{\textsc{u}} \in \mathscr N_2(\mbH')$. 
	
	Now, if $q=6$ there are two incomplete $3$-uniform graph in the sense of Theorem \ref{thm:suff3unifincomplet}: 
	the complete $3$-uniform hypergraph on $\llbracket 1,6 \rrbracket$ minus one hyperedge, say $\{1,2,3\}$; and the 
	complete $3$-uniform hypergraph on $\llbracket 1,6 \rrbracket$ minus two disjoint hyperedges, say $\{1,2,3\}$ and $\{4,5,6\}$. In both cases, $\{4,5,6\}$ is a minimal transversal of $\mbH'$, thus $\tau(\mbH')=3$, and so 
	$\mu_{\textsc{u}}(T')\ge 3/6 > 1/3,$ for all $T'\in\mathcal T(\mbH)$, 
	proving again that $\mu_{\textsc{u}} \in \mathscr N_2(\mbH')$. 
	
	We now address the case where $q>6$. First observe that 
	\begin{equation}
	\label{eq:trampo1}
	 q-2-\left\lfloor {q\over 3}\right\rfloor > {q\over 3}.
	\end{equation}
	Then, let $p=|\mathcal J|$ (using the notation of Theorem \ref{thm:suff3unifincomplet}), 
	and denote 
	$\mathcal J=\{H_1,...,H_p\}$. It is easily seen that a transversal of $\mbH$ can be constructed 
	from any minimal transversal of $\mbH'$, by induction, as follows: 
	\begin{itemize}
	\item Take a minimal transversal $T'$ of $\mbH'$, 
	and set $\maH_0:=\maH'$ and $T_0:=T'$. 
	\item For any  $i=1,...,p$, set $\maH_i = \maH_{i-1} \cup \{H_i\}$ and set $T_i$, a 
	transversal of $(V,\maH_i)$ of minimal size among those including 
	$T_{i-1}$.  (${T_{i}}$ necessarily exists since $T_{i-1}\cup \{H_i\}$ is a transversal of $(V,\maH_{i})$, as easily seen by induction.) 
	\item We obtain $\maH=\maH_p$ by construction, and $T:=T_p$ is a transversal of $\mbH$. 
	\end{itemize}
	We claim that 
	\begin{equation}
	\label{eq:trampo2}
	|T| \le |T'|+p.
	\end{equation}
	To see this, observe that for any $i=1,...,p$ we have the following alternative: either $H_i \cap T_{i-1} = \emptyset$, in which case we can take $T_i $ of the form $T_{i-1}\cup \{k\}$ for any $k\in H_i$, or 
	$H_i \cap T_{i-1} \ne  \emptyset$, in which case $T_i = T_{i-1}$. In all cases we have that $|T_i| \le |T_{i-1}| + 1$, and (\ref{eq:trampo2}) follows by induction. 
	Observing that $|T| \ge \tau(H)=q-2$, that, as the $H_i$'s are disjoint, 
	$p\le \lfloor {q\over 3}\rfloor$,  and using (\ref{eq:trampo1}) and (\ref{eq:trampo2}), we finally obtain that 
	\[\mu_{\textsc{u}}(T') = {|T'| \over q} \ge {|T|-p\over q} >{1\over 3},\]
	hence, once again $\mu_{\textsc{u}} \in \mathscr N_2(\mbH')$. 
	
	To conclude, $\mu_{\textsc{u}}$ is in all cases, an element of $\mathscr S_1(\mbH')$, implying that the model $(\mathbb{H}',\textsc{ml},\mu_{\textsc{u}})$ is stable for all such $\mbH'$. 
	\end{ex}

\section{Conclusion and perspectives}
\label{sec:conclu}
In this paper, we have introduced a new declination of stochastic matching models, by allowing the matching structure to be an hypergraph, rather than just a graph. 
By doing so, we generalize the recently introduced general stochastic matching models (GM). This class of models appears to have a wide range of applications in operations management, healthcare and assemble-to-order systems. 
After formally introducing the model, we have proposed a simple Markovian representation, under IID assumptions. 
We have then addressed the general question of stochastic stability, viewed as the positive recurrence of the underlying Markov chain. For this class of systems, solving this elementary and central question turns out to be an intricate problem. 
As the results of Sections \ref{sec:Ncond} and \ref{sec:instable} demonstrate, stochastic matching models on hypergraphs are in general, difficult to stabilize. Unlike the GM on graphs, the non-emptiness of the stability region on matching models on hypergraphs depends on a collection of conditions on the 
geometry of the compatibility hypergraph: rank, anti-rank, degree, size of the transversals, existence of cycles, and so on. 

Nevertheless, we show in Section \ref{sec:stable} that the 'house' of stable systems is not empty, but shelters models on 
various uniform hypergraphs that are complete, or complete up to a partition of their nodes (which is a reasonable assumption regarding kidney exchange programs with $3$-cycles, in which case, according to the compatibility of blood types and immunological characteristics, most but not all hyperedges of size $3$ appear in the compatibility graph). We provide the exact stability region of the system in the first case, and a lower bound in the second. For this we resort to ad-hoc 
multi-dimensional Lyapunov techniques. 

There is still much to do regarding this class of systems. Providing the precise stability region of 
a wider class of systems appearing in other applications is a tedious task, and is the subject of our ongoing research on this topic. 
As the present results tend to demonstrate, such advances are likely to be obtained only on a case-by-case basis. 
To go beyond the stability study, crucial questions of interest are, among others:  
the performance evaluation in steady state and a qualitative comparison of systems and matching policies. 
We believe that the present work thus represents a good starting point for a fruitful avenue on research on such systems.

\end{document}